\documentclass[final]{siamltex}

\usepackage{makeidx}         % allows index generation
\usepackage{graphicx}        % standard LaTeX graphics tool
\usepackage{multicol}        % used for the two-column index
\usepackage{amsmath}
\usepackage{amssymb}
\usepackage{latexsym}
\usepackage{color}
\usepackage[latin1]{inputenc}
\usepackage{tikz}
\usetikzlibrary{trees}
\newtheorem{remark}{Remark}[section]
\def\ds{\displaystyle}
\def\e{{\varepsilon}}
\def\div{\mathop{\mathrm{div}}}

\title{DATA FUSION OF SATELLITE IMAGERY FOR GENERATION OF DAILY CLOUD-FREE IMAGES AT HIGH RESOLUTION LEVEL}

\author{Natalya Ivanchuk\thanks{Department of Computer Sciences and Applied Mathematics, National University of Water and Environmental Engineering, Soborna str., 11, Rivne, Ukraine, 33028; EOS Data Analytics Ukraine, Desyatynny lane, 5, 01001 Kyiv, Ukraine} ({\tt n.v.medvid@nuwm.edu.ua, natalya.ivanchuk@eosda.com}) \and 	
	Peter Kogut\thanks{Department of Differential Equations, Oles Honchar Dnipro National University, Gagarin av., 72, 49010 Dnipro,
Ukraine; EOS Data Analytics Ukraine, Gagarin av., 103a, Dnipro, Ukraine} ({\tt p.kogut@i.ua, peter.kogut@eosda.com}) \and Petro Martyniuk\thanks{Institute of Automation, Cybernetics and Computer Engineering, National University of Water and Environmental Engineering, Soborna str., 11, Rivne, Ukraine, 33028; EOS Data Analytics Ukraine, Gagarin av., 103a, Dnipro, Ukraine} ({\tt {p.m.martyniuk@nuwm.edu.ua, petro.martyniyk@eosda.com}} )}
\begin{document}

\maketitle

\begin{abstract}
	In this paper we discuss a new variational approach to the Date Fusion problem of multi-spectral satellite images from Sentinel-2 and MODIS that have been captureed at different resolution level and, arguably, on different days. The crucial point of our approach that the MODIS image is cloud-free whereas the images from Sentinel-2 can be corrupted by clouds or noise.
\end{abstract}



\pagestyle{myheadings} \thispagestyle{plain} \markboth{N. Ivanchuk, P. Kogut, P. Martyniuk} {Data Fusion of Sateellite Imagery}

\section{Introduction}
\label{Sec_0}

Following \cite{Jo}, the Image Fusion is a process of combining the relevant information from a set of images of the same scene into a single image and the resultant fused image must be more
informative and complete than any of the input images. At the same time, when we deal with the data fusion problem for satellite images, such images, as a rule, are multi sensor, multi-modal,
multi-focus and multi temporal. Moreover, the data fusion problem is often exacerbated by cloud contamination. 
In some cloudy areas, researchers are fortunate to get 2--3 cloud-free satellite scenes per year, what is insufficient for many applications that require
dense temporal information, such as crop condition monitoring and phenology studies \cite{Gao,Wang}. In view of this, we can indicate the following general requirements for the satellite image fusion process:
\begin{itemize}
	\item[(i)] The fused image should preserve all relevant
	information from the input images;
	\item[(ii)] The image fusion should not introduce artifacts which
	can lead to wrong inferences.
\end{itemize}

In spite of the fact that the first requirement (item (i)) sounds rather vague, we give a precise treatment for it in Section~\ref{Sec 5}, making use of a collection of special constrained minimization problems (see \eqref{5.5aa}--\eqref{5.5ab}). As for the second item, it is important to emphasize that we are mainly interesting by satellite images that can be useful from agricultural point of view (land cover change mapping, crop
condition monitoring, yield estimation, and many others). Because of this  an important option in the image data fusion  is to preserve the precise geo-location of the existing crop fields and 
avoid an appearance of the so-called false contours and pseudo-boundaries on a given territory.

In this paper we mainly focus on the image fusion problem coming from two satellites --- Satnitel-2 and Moderate Resolution Imaging Spectroradiometer (MODIS). Since each band (spectral channel) in Sentinel images has $10$, $20$, or $60$ meters in pixel size, it gives an ideal spatial resolution for vegetation mapping at the field scale. Moreover, taking into account that Sentinel-2 has $3$--$5$ revisit cycle over the same territory, it makes its usage  for studying global biophysical processes, which allows to evolve rapidly during the growing season, essentially important. The unique problem that drastically restricts 
its practical implementation, is the fact that the satellite images, as a rule, are often contaminated by clouds, shadows, dust, and other atmospheric artifacts.

One of possible solutions for practical applications is to make use of frequent coarse-resolution data of the MODIS. Taking into account that the MODIS data can be delivered with the daily repeat cycle and 500-m surface reflectance, the core idea is to use the Sentinel and cloud-free MODIS data to generate synthetic 'daily' surface reflectance products at Sentinel spatial resolution \cite{KKU}.

The problem we consider in this paper can be briefly described as follows. We have a collection of multi-band images  $\left\{S_1, S_2,\dots, S_N:G_H\to\mathbb{R}^m\right\}$ from Sentinel-2 that were captured at some time instances $\left\{t_1,t_2,\dots,t_N\right\}$, respectively, and we have a MODIS image $M:G_L\to\mathbb{R}^n$ from some day $t_M$. It is assumed that all of these images are well co-registered with respect to the unique geographic location. We also suppose that the MODIS image is cloud-free and the day $t_M$ may does not coincide with any of time instances $\left\{t_1,t_2,\dots,t_N\right\}$. Meanwhile, the Sentinel images $\left\{S_2,\dots, S_N:G_H\to\mathbb{R}^m\right\}$ can be corrupted by clouds. The main question is how to generate a new synthetic 'daily' multi-band image of the same territory from the day $t_M$ at the Sentinel-2 spatial resolution $G_H$, utilizing for that the above mentioned data.

In principle, this problem is not new in the literature (see, for instance, \cite{Gao,Hilker,Ju,Masek,Wang}). For nowadays the spatial and temporal adaptive reflectance fusion model
(STARFM) is one of the most popular model where the idea to generate a new synthetic 'daily' satellite images at high resolution level has been realized. Basing on a deterministic weighting
function computed by spectral similarity, temporal difference, and spatial distance, this model (as many other generalizations) allows to predict daily surface
reflectance at high spatial resolution and MODIS temporal frequency. However, its performance essentially depends on the characteristic patch size of the landscape
and degrades somewhat when used on extremely heterogeneous
fine-grained landscapes \cite{Gao}. 

As it was mentioned in \cite{Bung1}, the majority image interpolation techniques are essentially based on certain assumptions regarding the data and the corresponding imaging devices. First, it is often assumed that the image
of high spatial resolution is a linear combination of the spectral channels with known weights \cite{Loncan}. Second, the loss of resolution is usually modeled as a linear operator which consists of a subsampled convolution with known kernel (point spread function). While both assumptions may be justified in some applications, it may be difficult to measure
or estimate the weights and the convolution kernel in a practical situation.

Instead of this,  we mainly focus on the variational approach to the satellite image data fusion. In order to eliminate the above mentioned restrictions,  we formulate the data fusion problem as the two-level optimization problem. At the first level, following a simple iterative procedure, we generate the so-called structural prototype for a synthetic Sentinel image from the given day $t_M$. The main characteristic feature of this prototype is the fact that, it must have a similar geometrical structure (namely, precise location of contours and field boundaries) to the nearest in time 'visible' Sentinel images, albeit they may have rather different intensities in all bands. Since the revisit cycle of Sentinel-2 is $3$--$5$ days, such prototype can be easily generated (for the details, we refer to Section~\ref{Sec 4}). 
In fact, we consider the above mentioned structural prototype as a reasonable input data for 'daily' prediction problem that we  formulate in the form of a special constrained minimization problem,  where the cost functional has a nonstandard growth  and the edge information for restoration of MODIS cloud-free images at the Sentinel resolution is accumulated both in the variable exponent of nonlinearity and in the directional image gradients which we derive from the predicted structural prototype. 

Our approach is based on the variational model in Sobolev-Orlicz space with a non-standard growth condition of the objective functional and on the assumption that, to a large extent, the image topology in each spectral channel is similar to the 
topographic map of its structural prototype. It is worth to emphasize that this model is considerably different from the variational model for P+XS image fusion that was proposed in \cite{Ballester}.  We discuss the
well thoroughness of the above approach and consistency of the corresponding variational problem, and we show that this problem admits a unique solution. We also derive some optimality conditions and supply our approach by results of numerical simulations with the real satellite images.

The paper is organized as follows. Section~\ref{Sec 1} contains some preliminaries, auxiliary results, and a non-formal statement of the data fusion problem.  In Section~\ref{Sec 3} we begin with some key assumptions and after we give a precise statement of the satellite image data fusion in the form of two-level constrained optimization problem with a nonstandard growth energy functional.
We show that, in principle, we can distinguish three different statements of the data fusion problem. Namely, it is the so-called restoration problem (when the main point is to restore the information in the cloud-corrupted zone for Sentinel images), the interpolation problem (when the day $t_M$ is intermediate for some time instances of $\left\{t_1,t_2,\dots,t_N\right\}$), and the so-called extrapolation problem (when $t_M>t_N$). We also discuss the specific of each of these problems and their rigorous mathematical description. Section~\ref{Sec 4} is devoted to the study of a particular model for the prediction of structural prototypes. We also illustrate this approach by some numerical simulations. 
 
Consistency issues of the proposed minimization problems, optimality conditions, and their substantiation are studied in Sections~\ref{Sec 5} and \ref{Sec 7}.
For illustration of this approach, we give in Section~\ref{Sec 8} some results of numerical experiments
with real satellite images. 
The experiments undertaken in this study confirmed the efficacy of the proposed method and revealed that it can acquire plausible visual performance and satisfactory quantitative accuracy for agro-scenes with rather complicated texture of background surface. In Appendix we give the main auxiliary results concerning the Orlicz spaces and the Sobolev-Orlicz spaces with variable exponent.

\section{Non-Formal Statement of the Problem}
\label{Sec 1}

Let $\Omega\subset\mathbb{R}^2$ be a bounded connected open set with  a sufficiently smooth boundary $\partial\Omega$ and nonzero Lebesgue measure. In majority cases $\Omega$ can be interpreted as a rectangle domain. Let $G_H$ and $G_L$ be two sample grids on $\Omega$ such that $G_H=\widehat{G}_H\cap\Omega$ and $G_H=\widehat{G}_H\cap\Omega$, where
\begin{align*}
\widehat{G}_H&=\left\{(x_i,y_j)\left|  
\begin{array}{c}
x_1=x_H,\ 
x_i=x_1+\Delta_{H,x}(i-1),\ i=1,\dots,N_x,\\
y_1=y_H,\ y_j=y_1+\Delta_{H,y}(j-1),\  j=1,\dots,N_y,
\end{array}
\right.\right\},\\[1ex]
\widehat{G}_L&=\left\{(x_i,y_j)\left|  
\begin{array}{c}
x_1=x_L,\ 
x_i=x_1+\Delta_{L,x}(i-1),\ i=1,\dots,M_x,\\
y_1=y_L,\ 
y_j=y_1+\Delta_{L,y}(j-1),\  j=1,\dots,M_y,
\end{array}
\right.\right\},
\end{align*}
with some fixed points $(x_H,y_H)$ and $(x_L,y_L)$. Here, it is assumed that $N_x>> M_x$ and $N_y>> M_y$. 

Let $[0,T]$ be a given time interval. Normally, by $T$ we mean a number of days. Let $t_M$ and $\left\{t_k\right\}_{k=1}^N$ be  moments in time (particular days) such that $0\le t_1< t_2<\dots< t_N\le T$ and  $t_1<t_M<T$. Let $\left\{S_1, S_2,\dots, S_N:G_H\to\mathbb{R}^m\right\}$ be a collection of multispectral images of some territory, delivered 
from Sentinel-2, that were taken at time instances $t_1,t_2,\dots,t_N$, respectively. Hereinaftre, $m=13$ and it stands for the number of spectral channels in images from Sentinel-2.  Let $M:G_L\to\mathbb{R}^n$, with $n=6$, be a MODIS image of the same territory and this image has been captured at time $t=t_M$.  It is assumed that:
\begin{enumerate}
	\item[(i)] The Sentinel-2 images $S_k:G_H\to\mathbb{R}^m$, $k=2,\dots,N$ can be corrupted by some noise, clouds and blur. However, the first one 
	$S_1:G_H\to\mathbb{R}^m$ is a cloud-free image; 
	\item[(ii)] For further convenience we divide the set of all bands for Sentinel images onto
	two parts $J_1$ and $J_2$ with $\mathrm{dim}(J_1)=6$ and $\mathrm{dim}(J_2)=7$, where  
	\begin{table}[h!]
		\begin{center}
		\caption{Spectral channels of $J_1$}
		\label{Tab 1.1}
		 \begin{tabular}{|l|c|c|l|}
		 	\hline
		 	\textbf{Band} & \textbf{Resolution} & \textbf{Central Wavelenth} & \textbf{Description}\\
		 	\hline
		 	$B_2$ & $10$ m & $490$ m & Blue\\
		 	$B_3$ & $10$ m & $560$ m & Green\\
		 	$B_4$ & $10$ m & $665$ m & Red\\
		 	$B_{8a}$ & $20$ m & $865$ m & Visible and Near Infrared (VNIR)\\
		 	$B_{11}$ & $20$ m & $1610$ m & Short Wave Infrared (SWIR)\\
		 	$B_{12}$ & $20$ m & $2190$ m & Short Wave Infrared (SWIR)\\
		 	\hline
		 \end{tabular}
	    \end{center}
	\end{table}
	
	\begin{table}[h!]
		\begin{center}
			\caption{Spectral channels of $J_2$}
			\label{Tab 1.2}
			\begin{tabular}{|l|c|c|l|}
				\hline
				\textbf{Band} & \textbf{Resolution} & \textbf{Central Wavelenth} & \textbf{Description}\\
				\hline
				$B_1$ & $60$ m & $443$ m & Ultra Blue (Coastal and Aerosol)\\
				$B_5$ & $20$ m & $705$ m & Visible and Near Infrared (VNIR)\\
				$B_6$ & $20$ m & $740$ m & Visible and Near Infrared (VNIR)\\
				$B_7$ & $20$ m & $783$ m & Visible and Near Infrared (VNIR)\\
				$B_8$ & $10$ m & $842$ m & Visible and Near Infrared (VNIR)\\
				$B_9$ & $60$ m & $940$ m & Short Wave Infrared (SWIR)\\
				$B_{10}$ & $60$ m & $1375$ m & Short Wave Infrared (SWIR)\\
				\hline
			\end{tabular}
		\end{center}
	\end{table}
	
\item[(iii)] Each spectral channel of the MODIS image $M=\left[M_1,M_2,\dots,M_6\right]:G_L\to\mathbb{R}^6$ has the similar spectral characteristics to the corresponding channel of $J_1$-group $\left\{B_2, B_3, B_4, B_{8a}, B_{11}, B_{12}\right\}$, respectively;

\item[(iv)] The principle point is that the MODIS image $M:G_L\to\mathbb{R}^6$ is visually sufficiently clear and does not corrupted by clouds or its damage zone can be neglected; 

\item[(v)] The MODIS image $M:G_L\to\mathbb{R}^6$  and the  images $\left\{S_1, S_2,\dots, S_N:G_H\to\mathbb{R}^m\right\}$ from Sentinel-2 are rigidly co-registered. This means that the MODIS image after arguably some affine transformation and 
each Sentinel images after the resampling to the grid with low resolution $G_L$, could be successfully matched according to the unique geographic location. 

In practice, the co-registration procedure can be realized using, for instance, the open-source LSReg v2.0.2 software \cite{Roy, Yan} that has been used in a number of recent studies \cite{Frantz, Roy1}, or the rigid co-registration approach that has been recently developed in \cite{Kogut1,Kogut2}. However, in both cases, in order to find an appropriate affine transformation, we propose to apply this procedure not to the original images, but rather to the contour's map of their spectral energies
$Y_{M}:G_L\to \mathbb{R}$ and $Y_{S_j}:G_H\to\mathbb{R}$, 
where the last ones should be previously resampled to the grid of the low resolution $G_L$. Here,
\begin{align*}
	Y_{M}(z)&:=\alpha_1 M_1(z)+\alpha_2 M_2(z)+\alpha_3 M_3(z),\quad\forall\,z=(x,y)\in G_L,\\
	Y_{S_i}(z)&:=\alpha_1 S_{i,1}(z)+\alpha_2 S_{i,2}(z)+\alpha_3 S_{i,3}(z),\quad\forall\,z=(x,y)\in G_H
\end{align*}
with $\alpha_1=0.114$, $\alpha_2=0.587$, and $\alpha_3=0.299$.
\end{enumerate}

\begin{remark}
	Let us mention that in the case of digital images, the only accessible information is a sampled and quantized version of $I:\Omega\to\mathbb{R}^m$, $I(x_i,y_j)$,  where $\left\{(x_i,y_j)\in\Omega\right\}$ is a set of discrete points  and for each spectral channel $k=1,\dots,m$, $I_k(x_i,y_j)$ belongs in fact to a discrete set of values, $0,1,\dots, 255$ in many cases. Due to the Shannon's theory, it is plausible to assume that $I_k$ is recoverable at any point $(x,y)\in\Omega$ from
	the samples $I_k(x_i,y_j)$. So, we may assume that the image $I$ is known in a continuous domain, up to
	the quantization noise (see \cite{Ballester}). However, in practice, such reconstruction is not a trivial problem. 
\end{remark}

\subsection{Functional Spaces}

Let us recall some useful notations. 
For vectors $\xi\in\mathbb{R}^2$ and $\eta\in\mathbb{R}^2$, $\left(\xi,\eta\right)=\xi^t\eta$ denotes the standard vector inner product in $\mathbb{R}^2$, where $^t$ stands for the transpose operator. The norm $|\xi|$ is the Euclidean norm given by $|\xi|=\sqrt{(\xi,\xi)}$. Let $\Omega\subset\mathbb{R}^2$ be a bounded open set with a Lipschitz boundary $\partial\Omega$. For any subset $E\subset\Omega$ we
denote by $|E|$ its $2$-dimensional Lebesgue measure
$\mathcal{L}^2(E)$. Let $\overline{E}$ denote the closure of $E$, and $\partial E$ stands for its boundary. We define the characteristic function $\chi_E$ of $E$ by
\[
\chi_E(x):=\left\{
\begin{array}{ll}
1,&\ \text{for }\ x\in E,\\
0,&\ \text{otherwise}.
\end{array}
\right.
\]

Let $X$ denote a real Banach space with norm $\|\cdot\|_X$, and let $X^\prime$ be its dual. Let $\left<\cdot,\cdot\right>_{X^\prime;X}$ be the duality form on $X^\prime\times X$. By $\rightharpoonup$ and $\stackrel{\ast}{\rightharpoonup}$ we denote the weak and weak$^\ast$ convergence in normed spaces, respectively.

For given $1\le p\le +\infty$, the space $L^p(\Omega;\mathbb{R}^2)$ is defined by
$$L^p(\Omega;\mathbb{R}^2)=\left\{f:\Omega\rightarrow \mathbb{R}^2\ :\ \|f\|_{L^p(\Omega;\mathbb{R}^2)}<+\infty\right\},
$$
where
$\|f\|_{L^p(\Omega;\mathbb{R}^2)}=\left(\int_{\Omega} \lvert f(x)\rvert^p\,dx\right)^{1/p}$ for $1\le p<+\infty$.
The inner product of two functions $f$ and $g$ in $L^p(\Omega;\mathbb{R}^2)$ with $p\in[1,\infty)$ is given by
$$
\left(f,g\right)_{L^p(\Omega;\mathbb{R}^2)}=\int_{\Omega} \left(f(x),g(x)\right)\,dx=\int_{\Omega} \sum_{k=1}^2 f_k(x) g_k(x)\,dx.
$$

We denote by $C_c^\infty(\mathbb{R}^2)$ a locally convex space of all infinitely differentiable functions with compact support in $\mathbb{R}^2$.
We recall here some functional spaces that will be used throughout this
paper. We define the Banach space $H^{1}(\Omega)$ as the closure of $C^\infty_c(\mathbb{R}^2)$ with respect to the norm
$$\|y\|_{H^{1}(\Omega)}=\left(\int_{\Omega} \left( y^2+\lvert\nabla y\rvert^2\right)\,dx\right)^{1/2}.$$
We denote by $\left(H^{1}(\Omega)\right)^\prime$ the dual space of $H^{1}(\Omega)$.  

Given a real Banach space $X$, we will denote by $C([0,T];X)$ the space of all continuous functions from $[0,T]$ into $X$. We recall that a function $u:[0,T]\to X$ is said to be Lebesgue measurable if there exists a sequence $\left\{u_k\right\}_{k\in\mathbb{N}}$ of step functions (i.e., $u_k=\sum_{j=1}^{n_k} a_j^k \chi_{A_j^k}$ for a finite number $n_k$ of Borel subsets $A_j^k\subset[0,T]$ and with $a_j^k\in X$) converging to $u$ almost everywhere with respect to the Lebesgue measure in $[0,T]$.

Then for $1\le p<\infty$, $L^p(0,T;X)$ is the space of all measurable functions $u:[0,T]\to X$ such that
\[
\|u\|_{L^p(0,T;X)}=\left(\int_0^T \|u(t)\|^p_X\,dt\right)^{\frac{1}{p}}<\infty,
\]
while $L^\infty(0,T;X)$ is the space of measurable functions such that
\[
\|u\|_{L^\infty(0,T;X)}=\sup_{t\in[0,T]}\|u(t)\|_X<\infty.
\]
The full presentation of this topic can be found in \cite{Lions_D}.

Let us recall that, for $1\le p\le \infty$, $L^p(0,T;X)$ is a Banach space. Moreover, if $X$ is separable and $1\le p<\infty$, then the dual space of $L^p(0,T;X)$ can be identified with $L^{p^\prime}(0,T;X^\prime)$.

For our purpose $X$ will mainly be either the Lebesgue space $L^p(\Omega)$ or $L^p(\Omega;\mathbb{R}^2)$ or the Sobolev space $W^{1,p}(\Omega)$ with $1\le p<\infty$.
Since, in this case, $X$ is separable, we have that $L^p(0,T;L^p(\Omega))=L^p(Q_T)$ is the ordinary Lebesgue space defined in $Q_T=(0,T)\times\Omega$. As for the space $L^p(0,T;W^{1,\alpha}(\Omega))$ with $1\le\alpha,p<+\infty$, it consists of all functions $u:[0,T]\times\Omega\to\mathbb{R}$ such that $u$ and $|\nabla u|$ belongs to $L^p(0,T;L^\alpha(\Omega))$.   Moreover,
\[
\left(\int_0^T\left[\left(\int_\Omega |u|^\alpha\,dx\right)^\frac{p}{\alpha}+\left(\int_\Omega |\nabla u|^\alpha\,dx\right)^\frac{p}{\alpha}\right]\,dt\right)^\frac{1}{p}
\]
defines the norm in $L^p(0,T;W^{1,\alpha}(\Omega))$.

\subsection{Topographic Maps and Geometry of Satillite Multispectral Images}
\label{Subsec_1}
Following the main principle of the Mathematical Morphology, a scalar image $u:\Omega\rightarrow \mathbb{R}$ is a representative of an equivalence class of images $v$ obtained from $u$ via a contrast change, i.e., $v=F(u)$, where $F$ is a continuous strictly increasing function. Under this assumption, a scalar image can be characterized by its upper (or lower) level sets $Z_\lambda(u)=\left\{x\in\Omega\ :\ u(x)\ge\lambda\right\}$ (resp. $Z^\prime_\lambda(u)=\left\{x\in\Omega\ :\ u(x)\le\lambda\right\}$). Moreover, each image can 
be recovered from its level sets by the reconstruction formula
$u(x)=\sup\left\{\lambda\ :\ x\in Z_\lambda(u)\right\}$. 
Thus, according to the Mathematical Morphology Doctrine,
the reliable information in the image contains in the level
sets, independently of their actual levels (see \cite{CCM_99} for the details). So, we can suppose that the entire geometric information about a scalar image is contained in those level sets.

In order to describe the level sets by their boundaries, $\partial Z_\lambda(u)$, we assume that $u\in W^{1,1}(\Omega)$, where $W^{1,1}(\Omega)$ stands for the standard Sobolev space of all functions $u\in L^1(\Omega)$ with respect to the norm
\[
\|u\|_{W^{1,1}(\Omega)}=\|u\|_{L^1(\Omega)}+\|\nabla u\|_{L^1(\Omega)^2},
\]
where the distributional gradient $\nabla u=\left(\frac{\partial u}{\partial x_1}, \frac{\partial u}{\partial x_2}\right)$ is represented as follows
\[
\int_{\Omega}u\frac{\partial\phi}{\partial x_i}\,dx=-\int_{\Omega}\phi \frac{\partial u}{\partial x_i}\,dx,\quad\forall\,\phi\in C^\infty_0(\Omega),\ i=1,2.
\]

It was proven in \cite{Ambrosio} that if $u\in W^{1,1}(\Omega)$ then its upper level sets $Z_\lambda(u)$ are sets of finite perimeter. So, the boundaries $\partial Z_\lambda(u)$ of
level sets can be described by a countable family of Jordan
curves with finite length, i.e., by continuous maps from the circle into the plane $\mathbb{R}^2$ without crossing points. As a result, at almost all points of almost all level sets of $u\in W^{1,1}(\Omega)$ we may define a unit normal vector $\theta(x)$. This vector field formally satisfies the following relations
\[
\left(\theta, \nabla u\right) = |\nabla u|\quad\text{and}\quad
|\theta|\le 1 \ \text{a.e. in $\Omega$}.
\]
In the sequel, we will refer to the vector field $\theta$ as the vector field of unit normals to the topographic map of a function $u$. So, we can associate $\theta$ with the geometry of the scalar image $u$.

In the case of multi-band satellite images $I:\Omega\to \mathbb{R}^m$, we will impose further the following assumption: $I\in W^{1,1}(\Omega;\mathbb{R}^m)$ and each spectral  channel of a given image $I$ has the same geometry. We refer to \cite{CCM_02} for the experimental discussion.

\begin{remark}
	\label{Rem 3.1}
	In practice, at the discrete level, the vector field $\theta(x,y)$ can be defined by the rule $\theta(x_i,y_j)=\frac{\nabla u(x_i,y_j)}{|\nabla u(x_i,y_j)|} $ when  $\nabla u(x_i,y_j)\ne 0$, and $\theta=0$ when $\nabla u(x_i,y_j)=0$. However, as was mentioned in \cite{Ballester}, a better choice for $\theta(x,y)$ would be to
	compute it as the ration $\frac{\nabla U(t,\cdot)}{|\nabla U(t,\cdot)|}$ for some
	small value of $t > 0$, where $U(t,x,y)$ is a solution of the following initial-boundary value problem with $1D$-Laplace operator in the principle part
	\begin{align}
		\label{2.10}
		&\frac{\partial U}{\partial t}=\mathrm{div}\,\left(\frac{\nabla U}{|\nabla U|}\right),\quad   t\in(0,+\infty),\ (x,y)\in\Omega,\\
		\label{2.11}
		&U(0,x,y)= u(x,y),\quad (x,y)\in\Omega,\\
		\label{2.12}
		&\frac{\partial U(0,x,y)}{\partial\nu}=0,\quad t\in(0,+\infty),\ (x,y)\in\partial\Omega.
	\end{align}
	
	As a result, for any $t>0$, there can be found a vector field 
	$$
	\xi\in L^\infty(\Omega;\mathbb{R}^2)\ \text{ with }\  \|\xi(t)\|_{L^\infty(\Omega;\mathbb{R}^2)}\le 1
	$$ 
	such that
	\begin{equation}
		\label{2.13}
		\left(\xi(t), \nabla U(t,\cdot)\right)=|\nabla U(t,\cdot)|\ \text{in }\ \Omega,\quad 
		\xi(t)\cdot \nu=0\ \text{on }\ \partial\Omega,
	\end{equation}
	and $	U_t(t,x,y)=\mathrm{div}\, \xi(t,x,y)$ in the sense of distributions on $\Omega$ for a.a. $t>0$.
	
	We notice that in the framework of this procedure, for small value of $t > 0$, we do not distort the geometry of the function $u(x,y)$ in an essential way. Moreover, it can be shown that
	this regularization of the vector field $\theta(x,y)=\frac{\nabla U(x,y)}{|\nabla U(x,y)|}$ satisfies condition 
	$\mathrm{div}\,\theta \in L^2(\Omega)$.
\end{remark}

\subsection{Texture Index of a Gray-Scale Image}
Let $u\in C([0,T]; L^2(\Omega))$ be a given function. For each $t\in[0,T]$, we associate the real-valued mapping
$u(t,\cdot):\Omega\mapsto\mathbb{R}$ with a gray-scale image, and the mapping $u:(0,T)\times\Omega\rightarrow \mathbb{R}$ with an optical flow. A widely-used way to smooth $u(t,\cdot)$ is by calculating the convolution 
\[
u_\sigma(t,x):=\left(G_\sigma\ast \widetilde{u}(t,\cdot)\right)(x)=\int_{\mathbb{R}^2} G_\sigma(x-y) \widetilde{u}(t,y)\,dy,
\]
where $\widetilde{u}$ denotes zero extension of $u$ from $Q_T=(0,T)\times\Omega$ to $\mathbb{R}^3$, and 
$G_\sigma$ stands for the two-dimentional Gaussian of width (standard deviation) $\sigma>0$:
\[
G_\sigma(x)=\frac{1}{2\pi\sigma^2}\exp\left(-\frac{|x|^2}{2\sigma^2}\right).
\]

\begin{definition}
	\label{Def 1.2}
	We say that a function $p_u:(0,T)\times\Omega\to\mathbb{R}$ is the  texture index of a given optical flow $u:(0,T)\times\Omega\rightarrow \mathbb{R}$ if
	it is defined by the rule
	\begin{equation}
		\label{1.6}
		p_u(t,x):=1+g\left(\frac{1}{h}\int_{t-h}^t\left|\left(\nabla G_\sigma\ast \widetilde{u}(\tau,\cdot)\right) (x)\right|^2\,d\tau\right),\ \forall\,(t,x)\in Q_T,
	\end{equation}
where
$g{:}[0,\infty)\rightarrow (0,\infty)$ is the edge-stopping function which we take in the form of the Cauchy law $
g(s)=\frac{a}{a+s}$ with $a>0$ small enough, and $h>0$ is a small  positive value.
\end{definition}

Since $G_\sigma\in C^\infty(\mathbb{R}^2)$, it follows from \eqref{1.6} and absolute continuity of the Lebesgue integral that $1<p_u(t,x)\le 2$ in $Q_T$ and $p_u\in C^1([0,T];C^\infty(\mathbb{R}^2))$ even if $u$ is just an absolutely integrable function in $Q_T$. Moreover, for each $t\in[0,T]$, $p_u(t,x)\approx 1$ in those places of $\Omega$ where some edges or discontinuities are present in the image $u(t,\cdot)$, and $p_u(t,x)\approx 2$ in places where $u(t,\cdot)$ is smooth or contains homogeneous features. In view of this, $p_u(t,x)$ can be interpreted as a characteristic of the sparse texture of the function $u$ that can change with time.
The following result plays a crucial role in the sequel (for the proof we refer to \cite[Lemma 2.1]{Kogut2022}).
\begin{lemma}
	\label{Lemma 1.7}
	Let $u\in C([0,T];L^2(\Omega))$ be a measurable function  extended by zero outside of $Q_T$. Let 
	$$
	p_{u}=1+g\left(\frac{1}{h}\int_{t-h}^t\left|\left(\nabla G_\sigma\ast \widetilde{u}(\tau,\cdot)\right)\right|^2\,d\tau\right)
	$$  
	be the corresponding  texture index. Then there exists a constant $C>0$ depending on $\Omega$, $G$, and $ \|u\|_{C([0,T];L^2(\Omega))}$ such that
	\begin{gather}
		\label{1.7a}
		\alpha:=1+\delta\le p_{u}(t,x) \le \beta:=2,\quad \forall\,(t,x)\in Q_T, \\
		\label{1.7b}
		p_{u}\in C^{0,1}(Q_T),\ 
			|p_{u}(t,x)-p_{u}(s,y)|\le C\left(|x-y|+|t-s|\right),\ \forall\,(t,x),(s,y)\in \overline{Q_T},
	\end{gather}	
	where
	\begin{align}		
	\label{1.7d}
	\delta&={ah}\left[{ah+\|G_{\sigma}\|^2_{C^1(\overline{\Omega-\Omega})} |\Omega|\|u\|^2_{L^2(0,T;L^2(\Omega))}}\right]^{-1},\\ \|G_{\sigma}\|_{C^1(\overline{\Omega-\Omega})}&:=\max\limits_{z=x-y\atop x\in\overline{\Omega}, y\in\overline{\Omega}} \Big[|G_\sigma(z)|+|\nabla G_\sigma(z)|\Big]	\label{1.7e}
	=\frac{e^{-1}}{\left(\sqrt{2\pi}\sigma\right)^2}\left[1+\frac{1}{\sigma^2}\mathrm{diam}\,\Omega\right].
	\end{align}
\end{lemma}

\section{Data Fusion Problem. Main Requirements to the Formal Statement}
\label{Sec 3}

Let $\left\{S_1, S_2,\dots, S_N:G_H\to\mathbb{R}^m\right\}$, with $m=13$, be a collection of multispectral images of some territory
from Sentinel-2 that were taken at time instances $\left\{t_1,t_2,\dots,t_N\right\}\subset[0,T]$, respectively. We admit that these images can be corrupted because of
poor weather conditions, such as rain, clouds, fog, and dust conditions. Typically, the measure of degradation of optical satellite images can be such that we cannot even rely on some reliability of pixel's values inside the damaged regions for each of spectral channels. As a result, some subdomains of such images become absolutely invisible. Let $\left\{D_1, D_2,\dots, D_N\right\}\subset 2^\Omega$ be a collection of damage regions for the corresponding Sentinel-images. So, in fact, we deal with the set of images 
\[
\left\{S_i: G_H\setminus D_i\to\mathbb{R}^m,\ i=1,\dots,N\right\}.
\]

Let $M:G_L\to\mathbb{R}^n$, with $n=6$, be a MODIS image of the same territory and this image has been captured at time $t=t_M\in (t_1,T)$.

Before proceeding further, we begin with the following assumption:
\begin{enumerate}
	\item[(a)] $D_1=\emptyset$ and the damage zones for the rest images from Sentinel-2 are such that each $D_i$, $i=2,\dots,N$, is a measurable closed subset of $\Omega$ with property $\mathcal{L}^2(D_i)\le 0.6\,\mathcal{L}^2(\Omega)$, where $\mathcal{L}^2(D_i)$ stands for the $2$-D Lebesgue measure of $D_i$; 
	
	\item[(b)] The MODIS image $M:G_L\to\mathbb{R}^n$ is assumed to be cloud-free;
	
	\item[(c)] The images $M:G_L\to\mathbb{R}^6$  and   $\left\{S_1, S_2,\dots, S_N:G_H\to\mathbb{R}^m\right\}$ are rigidly co-registered. This means that the MODIS image after arguably some affine transformation and 
	each Sentinel images after the resampling to the grid with low resolution $G_L$, could be successfully matched according to the unique geographic location;
	
	\item[(d)] There exists an impulse response $\mathcal{K}$ such that, for any multi-band image with high resolution $I:G_H\rightarrow \mathbb{R}^m$, its resampling to the grid with low resolution $G_L$ can be expressed as follows
	\[
	I(x_i,y_j)=\left[\mathcal{K}\ast I\right](x_i,y_j),\quad \forall\,i=1,\dots,M_x,\ \forall\,j=1,\dots,M_y,
	\]
	where $\mathcal{K}\ast I$ stands for the convolution operator. For instance, setting $\mathcal{K}=[k_{p,q}]_{p, q=1,\dots,K}$, we have
	\[
	\left[\mathcal{K}\ast I\right](x_i,y_j)=\sum_{p=1}^K  \sum_{q=1}^K k_{p,q} I(x_{i-p+1},y_{j-q+1})
	\]	
	provided $I(x,y)=0$ if $(x,y)\not\in\Omega$. 
	 In majority of cases, it is enough to set up $k_{p,q}=\frac{1}{K^2}$, $\forall\,p,q=1,\dots,K$, with an appropriate choice of $K\in\mathbb{N}$. However, if we deal with satellite images containing some agricultural areas with medium sides fields of various shapes, then the more efficient way for the choice of kernel $\mathcal{K}=[k_{p,q}]_{p, q=1,\dots,K}$ is to define it using the weight coefficients of the Lanczos interpolation filters.
\end{enumerate}	

For our further analysis, we make use of the following notion. We say that the multi-band images $\left\{ \widehat{S}_i:G_H\rightarrow \mathbb{R}^m\right\}_{i=1}^N$
are structural prototypes of the corresponding cloud-corrupted ones $\left\{S_i: G_H\setminus D_i\to\mathbb{R}^m\right\}_{i=1}^N$ if they are defined as follows:
\begin{equation}
\label{2.3} 
\begin{split}
\widehat{S}_{1,k}(z)&=S_{1,k}(z),\\ 
\widehat{S}_{i,k}(z)&=\left\{
\begin{array}{ll}
S_{i,k}(z), & z\in G_H\setminus D_i,\\
\gamma_{i,k} \widehat{S}_{i-1,k}(z), & z\in G_H\cap D_i,
\end{array}
\right\},\ i=2,\dots,N,\ k=1,\dots,m,
\end{split}
\end{equation}
where
\[
\gamma_{i,k}=\frac{\Big[\chi_{\Omega\setminus D_i}S_{i,k}\Big]\cdot\Big[\chi_{\Omega\setminus D_i}\widehat{S}_{i-1,k}\Big]}{\big\|\chi_{\Omega\setminus D_i}\widehat{S}_{i-1,k}\big\|^2_{\mathcal{L}(\mathbb{R}^2,\mathbb{R}^2)}},
\]
$A\cdot B$ stands for the scalar product of two matrices $A$ and $B$, and $\|\cdot\|_{\mathcal{L}(\mathbb{R}^2,\mathbb{R}^2)}$ denotes the Euclidean norm of a matrix.  
Moreover, each structural prototype $\widehat{S}_i:G_H\rightarrow \mathbb{R}^m$ is rigidly related to the corresponding day $t_i$ when the image $S_i: G_H\setminus D_i\to\mathbb{R}^m$ had been captured.

\begin{remark}
	\label{Rem 2.3}
As follows from the rule \eqref{2.3}, this iterative procedure should be applied to each spectral channel of all multi-band images from Sentinel-2. Since the revisit time for Sentinel-2 is $3$--$5$ days and the collection of images $\left\{S_i: G_H\setminus D_i\to\mathbb{R}^m\right\}_{i=1}^N$ is rigidly co-registered, it follows from \eqref{2.3} that the structural prototypes $\left\{ \widehat{S}_i:G_H\rightarrow \mathbb{R}^m\right\}_{i=1}^N$ are also well co-registered and they have the similar topographic maps with respect to their precise space location, albeit some false contours can appear along the boundaries of the damage zones $D_i$. In fact, in order to avoid the appearance of the false contours, the weight coefficients $\gamma_{i,k}$ have been introduced.
\end{remark}

Since the MODIS image has been captured at a time instance $t_M\in  (t_1, T)$, we can have three possible cases: 
\begin{description}
	\item[(A1)] there exists an index $i^\ast\in\left\{1,2,\dots,N\right\}$ such that $t_M=t_{i^\ast}$;
	
	\item[(A2)] there exists an index $i^\ast\in\left\{1,2,\dots,N-1\right\}$ such that 
	$t_{i^\ast}<t_M<t_{i^\ast+1}$;
	
	\item[(A3)] $t_N<t_M<T$.
\end{description}

In view of this, we will distinguish three different statements of the data fusion problem:
\begin{description}
	\item[Case (A1)] (Restoration Problem) The problem (A1) consists in restoration of the damaged multi-band optical image $S_{i^\ast}:G_H\setminus D_{i^\ast}\to\mathbb{R}^m$ using result of its fusion with the cloud-free MODIS image $M:G_L\to\mathbb{R}^6$ of the same territory. It means that, we have to create a new image $S^{rest}_{i^\ast}:G_H\to\mathbb{R}^m$, which would be well defined on the entire grid $G_H$, such that
	\begin{gather}
	\label{2.4a}
	S^{rest}_{i^\ast}(z)=S_{i^\ast}(z),\quad \forall\,z=(x,y)\in G_H\setminus D_{i^\ast},\\
	\label{2.4b}
	\sum\limits_{z\in G_L\cap D_{i^\ast}}\Big( \left(\mathcal{K}\ast S^{rest}_{i,k}\right)(z)-M_{k}(z)\Big)^2=\inf\limits_{I\in\mathcal{I}} \sum\limits_{z\in G_L\cap D_{i^\ast}}\Big( \left(\mathcal{K}\ast I\right)(z)-M_{k}(z)\Big)^2,\atop \forall\,k\in J_1,\\
	\label{2.4c}
	S^{rest}_{i^\ast,k}(z)=\widehat{S}_{i^\ast,k}(z),\quad \forall\,z\in G_H,\ \forall\,k\in J_2.
	\end{gather}
	The precise description of the class of admissible (or feasible) images $\mathcal{I}$ will be given in the next section.
	
	\item[Case (A2)] (Interpolation Problem) The problem (A2) consists in generation of a new multi-band optical image $S^{int}_{t_M}:G_H\to\mathbb{R}^m$ at the Sentinel-level of resolution using result of the fusion of cloud-free MODIS image $M:G_L\to\mathbb{R}^6$ with the predicted structural prototype $\widehat{S}_{t_M}:G_H\rightarrow \mathbb{R}^m$ from the given day $t_M$. In fact, in this case we deal with the two-level problem. At the first level, having the collection of structural prototypes $\left\{ \widehat{S}_i:G_H\rightarrow \mathbb{R}^m\right\}_{i=1}^N$ which is associated with the time instances $\left\{t_1,t_2,\dots,t_N\right\}\subset[0,T]$, we create a new 'intermediate' image $\widehat{S}_{t_M}:G_H\rightarrow \mathbb{R}^m$ that can be considered as daily prediction of the topographical map of a given territory from the day $t_M$. Then, at the second level, we realize the fusion procedure of this predicted image with the cloud-free MODIS image $M:G_L\to\mathbb{R}^6$ of the same territory. As a result, we have to create a new image 
	$S^{interp}_{t_M}:G_H\to\mathbb{R}^m$ with properties:
	\begin{gather}
	\label{2.5a}
	\sum\limits_{z\in G_L}\Big( \left(\mathcal{K}\ast S^{int}_{t_M,k}\right)(z)-M_{k}(z)\Big)^2=\inf\limits_{I\in\mathcal{I}} \sum\limits_{z\in G_L}\Big( \left(\mathcal{K}\ast I\right)(z)-M_{k}(z)\Big)^2,\atop \forall\,k\in J_1,\\
	\label{2.5b}
	S^{int}_{t_M,k}(z)=\widehat{S}_{t_M,k}(z),\quad \forall\,z\in G_H,\ \forall\,k\in J_2.
	\end{gather}
	
	\item[Case (A3)] (Extrapolation Problem) The problem (A3) consists in generation of a new multi-band optical image $S^{ext}_{t_M}:G_H\to\mathbb{R}^m$ using result of the data assimilation from the cloud-free MODIS image $M:G_L\to\mathbb{R}^6$ into the structural prototype $\widehat{S}_N:G_H\rightarrow \mathbb{R}^m$ of the  Sentinel-image $S_N: G_H\setminus D_N\to\mathbb{R}^m$. Here, it is assumed that the level sets  of the given territory (or topographical map) for each Sentinel spectral channel from the day $t_M$ have the same geo-location as they have in $\widehat{S}_N:G_H\rightarrow \mathbb{R}^m$.
	So, we can set $\widehat{S}_{t_M}=\widehat{S}_N$. Thus, in the framework of this problem, we have to retrieve a new image $S^{ext}_{t_M}:G_H\to\mathbb{R}^m$, which would be well defined on the entire grid $G_H$, such that
	\begin{gather}
	\label{2.6a}
	\sum\limits_{z\in G_L}\Big( \left(\mathcal{K}\ast S^{ext}_{t_M,k}\right)(z)-M_{k}(z)\Big)^2=\inf\limits_{I\in\mathcal{I}}\sum\limits_{z\in G_L} \Big( \left(\mathcal{K}\ast I\right)(z)-M_{k}(z)\Big)^2,\atop
	 \forall\,k\in J_1,\\
	\label{2.6b}
	S^{ext}_{t_M,k}(z)=\widehat{S}_{t_M,k}(z),\quad \forall\,z\in G_H,\ \forall\,k\in J_2.
	\end{gather}
\end{description} 

To provide the detailed analysis of the above mentioned problems, we begin with some auxiliaries. 

\section{The Model for Prediction of Structural Prototypes}
\label{Sec 4}

Due to the iterative procedure \eqref{2.3}, we can define the so-called structural prototypes $\left\{ \widehat{S}_i:G_H\rightarrow \mathbb{R}^m\right\}_{i=1}^N$ for each cloud-corrupted Sentinel image $\left\{S_i: G_H\setminus D_i\to\mathbb{R}^m\right\}_{i=1}^N$. Let $i^\ast$ be an integer such that $t_{i^\ast}<t_M<t_{i^\ast+1}$. Let $\widehat{S}_{i^\ast,j}$ and $\widehat{S}_{i^\ast+1,j}$ be structural prototypes of the corresponding images from given days $t_{i^\ast}$ and $t_{i^\ast+1}$. Since $\widehat{S}_{i^\ast,j}$ and $\widehat{S}_{i^\ast+1,j}$ are well co-registered images, it is reasonable to assume that they have the similar geometric structure albeit they may have rather different intensities.

The main question we are going to discuss in this section is: how to correctly define the 'intermediate' image $\widehat{S}_{t_M}:G_H\rightarrow \mathbb{R}^m$ that can be considered as daily prediction of the topographical map of a given territory from the day $t_M$. With that in mind, for each spectral channel $j\in\left\{1,2,\dots,m\right\}$, we make use of the following model
\begin{gather}
	\label{4.1a}
	\ds\frac{\partial u}{\partial t}-\div \left(|\nabla u|^{p_u(t,x)-2}\nabla u\right)=v\quad\text{in }\ (t_{i^\ast},t_{i^\ast+1})\times\Omega,\\
	\label{4.1b}
	\partial_{\nu}u=0\quad\text{on }\ (t_{i^\ast},t_{i^\ast+1})\times\partial\Omega,\\
	\label{4.1c}
	u(t_{i^\ast},\cdot)=\widehat{S}_{i^\ast,j}(\cdot)\quad  \text{in }\ \Omega,
\end{gather}
where $p_u(t,x)$ stands for the texture index of the scalar image $u$ (see Definition~\ref{Def 1.2}), and $v\in L^2(t_{i^\ast},t_{i^\ast+1};L^2(\Omega))$ is an unknown source term that has to be defined in the way to guarantee the fulfillment (with some accuracy) of the relation
\begin{equation}
\label{4.2}
u(t_{i^\ast+1},\cdot)\approx \widehat{S}_{i^\ast+1,j}(\cdot)\quad\text{ in }\ \Omega.
\end{equation}
Here,  we assume that the images $\widehat{S}_{i^\ast,j}$ and $\widehat{S}_{i^\ast+1,j}$ in \eqref{4.1c} and \eqref{4.2} are well defined into the entire domain $\Omega$.

\begin{remark}
	\label{Rem 4.1} 
	The main characteristic feature of the proposed initial-boundary value problem (IBVP) is the fact that the exponent $p_u$ depend not only on $(t,x)$ but also on a solution $u(t,x)$ of this problem. It is well-known that the variable character of the exponent $p_u$ causes a gap between the monotonicity and coercivity conditions. Because of this gap, equations of the type \eqref{4.1a} can be termed as equations with nonstandard growth conditions. So, in fact, we deal with the Cauchy-Neumann IBVP for a parabolic equation of $p_u=p(t,x,u)$-Laplacian type with variable exponent of nonlinearity. It was recently shown that the model \eqref{4.1a}--\eqref{4.1c} naturally appears as the Euler-Lagrange equation in the problem of restoration of cloud contaminated satellite optical images \cite{DAKOMAUV2022,KhKU21}. Moreover, the above mentioned problem can be considered as a  model for the deblurring and denoising of multi-spectral images. In particular, this model has been proposed in \cite{DAKOMAUV2021,KKU} in order to avoid the blurring of edges and other localization problems presented by linear diffusion models in images processing. We also refer to \cite{KKM22_1}, where the authors study some optimal control problems associated with a special case of the model \eqref{4.1a}--\eqref{4.1c} and show that the given class of optimal control problems is well posed.
\end{remark}

Before proceeding further, we note that the distributed control $v$ in the right hand side of \eqref{4.1a} describes the fictitious sources or sinks of the intensity $u$ that may have a tendency to change at most pixels even for co-registered structural prototypes $\widehat{S}_{i^\ast,j}(\cdot)$ and ${S}_{i^\ast+1,j}(\cdot)$. As for the Neumann boundary condition $\partial_{\nu} u=0$ on $\partial\Omega$, this condition corresponds to the reflection of the image across the boundary and has the advantage of not imposing any value on the boundary and not creating 'edges' on it. So, it is very natural conditions if we assume that the boundary of the image is an arbitrary cutoff of a larger scene in view.

In order to characterize the solvability issues of the IBVP \eqref{4.1a}--\eqref{4.1c}, we adopt the following concept.

\begin{definition}
	\label{Def 4.4}
	We say that, for given $v\in L^2(\Omega)$ and $\widehat{S}_{i^\ast,j}\in L^2(\Omega)$, a function $u$ is a weak solution to the problem \eqref{4.1a}--\eqref{4.1c} if 
	\begin{equation}
	\label{4.4}
	u\in L^2(t_{i^\ast},t_{i^\ast+1};L^2(\Omega)),\  u(t,\cdot)\in W^{1,1}(\Omega)\ \text{ for a.a. $t\in[t_{i^\ast},t_{i^\ast+1}]$},\atop \ds  \int_{t_{i^\ast}}^{t_{i^\ast+1}}\int_{\Omega} |\nabla u|^{p_u(t,x)}\,dx dt<+\infty,
	\end{equation}
	and the integral identity
	\begin{multline}
	\label{4.5}
	\int_{t_{i^\ast}}^{t_{i^\ast+1}}\int_{\Omega} \left(-u\frac{\partial\varphi}{\partial t}+\left(|\nabla u|^{p_u}\nabla u,\nabla\varphi\right)\right)\,dx dt\\
	=\int_{t_{i^\ast}}^{t_{i^\ast+1}}\int_{\Omega} v\varphi\,dxdt+\int_\Omega \widehat{S}_{i^\ast,j}\varphi|_{t=t_{i^\ast}}\,dx	
	\end{multline}
	holds true for any function $\varphi\in \Phi$, where
	$\Phi=\left\{\varphi\in C^\infty([t_{i^\ast},t_{i^\ast+1}]\times\overline{\Omega})\ : \left.\varphi\right|_{t=t_{i^\ast+1}}=0\right\}$. 
\end{definition}

The following result highlights the way in what sense the weak solution takes the initial value $u(t_{i^\ast},\cdot)=\widehat{S}_{i^\ast,j}(\cdot)$.
\begin{proposition}[\cite{Kogut2022}]
	\label{Prop 2.6}
	Let $v\in H^1(\Omega)$ and $\widehat{S}_{i^\ast,j}\in L^2(\Omega)$ be given distributions.
	Let $u $ be a weak solution to the problem \eqref{4.1a}--\eqref{4.1c} in the sense of Definition~\ref{Def 4.4}. Then, for any $\eta\in C^\infty(\overline{\Omega})$, the scalar function $h(t)=\ds\int_\Omega u(t,x) \eta(x)\,dx$ belongs to $W^{1,1}(t_{i^\ast},t_{i^\ast+1})$ and $h(0)=\ds\int_\Omega \widehat{S}_{i^\ast,j}(x) \eta(x)\,dx$.
\end{proposition}

Utilizing the perturbation technique and the classical fixed point theorem of Schauder \cite{NIRENBERG}, it has been recently proven  the following existence result. 

\begin{theorem}[\cite{Kogut2022}]
	\label{Th 4.4}
	Let $v\in L^2(\Omega)$ and $\widehat{S}_{i^\ast,j}\in L^2(\Omega)$  be given distributions.
	Then the initial-boundary value problem \eqref{4.1a}--\eqref{4.1c} admits at least one weak solution $u=u(t,x)$ with the following higher inegrability properties 
	\begin{equation}
	\label{4.6}
	u\in L^\infty(t_{i^\ast},t_{i^\ast+1};L^2(\Omega)),\  u\in W^{1,\alpha}(\left(t_{i^\ast},t_{i^\ast+1}\right)\times\Omega),\atop \ds  u\in L^{2\alpha}\left(t_{i^\ast},t_{i^\ast+1};L^{2\alpha}(\Omega)\right),
	\end{equation}
	where the exponent $\alpha$ is given by the rule
	\begin{align*}
	\alpha&=
	{ah}\left[{ah+\|G_{\sigma}\|^2_{C^1(\overline{\Omega-\Omega})} |\Omega| \left(\|v\|^2_{L^2(Q_T)}+2\|\widehat{S}_{i^\ast,j}\|^2_{L^2(\Omega)}\right)}\right]^{-1}
	\end{align*}
\end{theorem} 

In order to satisfy the condition \eqref{4.2} and define an appropriate source term $v=v(t,x)$, 
we utilize some issues coming from the well-known method of Horn and Schunck \cite{Horn} that has been developed in order to compute optical flow velocity from spatiotemporal derivatives of image intensity. Following this approach, we define the function $v^\ast$ as a solution of the problem
\begin{multline}
\int_{\Omega}\left(\left.\frac{\partial Y}{\partial t}\right|_{t=(t_{i^\ast}+t_{i^\ast+1})/2}-\left.\div\left(|\nabla Y|^{p_Y}\nabla Y\right)\right|_{t=(t_{i^\ast}+t_{i^\ast+1})/2}-v\right)^2\,dx\\ +\lambda_1^2\int_{\Omega}|\nabla v|^2\,dx\rightarrow \inf_{v\in H^1(\Omega)},
\label{4.3}
\end{multline}
where $\lambda_1>0$ is a tuning parameter (for numerical simulations we take $\lambda_1=0.5$), and the spatiotemporal derivatives are computed by the rules
\begin{align*}
\left.\frac{\partial Y}{\partial t}\right|_{t=(t_{i^\ast}+t_{i^\ast+1})/2}&=\frac{1}{t_{i^\ast+1}-t_{i^\ast}} \left(\widehat{S}_{i^\ast+1,j}-\widehat{S}_{i^\ast,j}\right),\\
\left.\div\left(|\nabla Y|^{p_Y}\nabla Y\right)\right|_{t=(t_{i^\ast}+t_{i^\ast+1})/2}&=\frac{1}{2} \Big[\div\left( |\nabla \widehat{S}_{i^\ast,j}|^{p_{\widehat{S}_{i^\ast,j}}}\nabla \widehat{S}_{i^\ast,j}\right)\\
&\quad+\div\left(|\nabla \widehat{S}_{i^\ast+1,j}|^{p_{\widehat{S}_{i^\ast+1,j}}}\nabla \widehat{S}_{i^\ast+1,j}\right)\Big].
\end{align*}
 
It is clear that a minimum point $v^\ast\in H^1(\Omega)$ to unconstrained minimization problem \eqref{4.3} is unique and satisfies necessarily the Euler-Lagrange equation
\begin{gather}
\label{4.7a}
\lambda_1^2\Delta v^\ast +\left(\left.\frac{\partial Y}{\partial t}\right|_{t=(t_{i^\ast}+t_{i^\ast+1})/2}-\left.\div\left(|\nabla Y|^{p_Y}\nabla Y\right)\right|_{t=(t_{i^\ast}+t_{i^\ast+1})/2}-v^\ast\right)=0
\end{gather}
with the Nuemann boundary condition $\partial_{\nu} v^\ast=0$ on $\partial\Omega$.

Setting $v=v^\ast$ in \eqref{4.1a}, we can define a function $u^\ast=u^\ast(t,x)$ as the weak solution of the IBVP \eqref{4.1a}--\eqref{4.1c}. Numerical experiments show that, following this way, we obtain a function $u^\ast$ with properties \eqref{4.4} and \eqref{4.6} such that
\[
u^\ast(t_{i^\ast},x)=\widehat{S}_{i^\ast,j}(x)\quad\text{and}\quad u^\ast(t_{i^\ast+1},x)\approx \widehat{S}_{i^\ast+1,j}(x)\quad\text{in }\ \Omega,
\]
where the peak signal-to-noise ratio (PSNR) between images $u^\ast(t_{i^\ast+1},x)$ and $\widehat{S}_{i^\ast+1,j}(x)$ is sufficiently large, $PSNR> 46$.

This observation leads us to the following conclusion: the 'intermediate' image $\widehat{S}_{t_M}:G_H\rightarrow \mathbb{R}^m$
can be defined as follow:
\begin{equation}
\label{4.7}
\widehat{S}_{t_M}(x)=u^\ast(t_{M},x),\quad \forall\, x\in G_H.
\end{equation}

To illustrate how the proposed model \eqref{4.1a}--\eqref{4.1c} works, we consider as an input data two images $S_1$, $S_2$ of some region that represent a typical agricultural area in Australia the resolution $20 m/pixel$. These images have been delivered from Sentinel-2 and captured at the time instances $t_1='July, 08'$ and $t_2='August, 25'$, respectively. Both of these images are cloud-free (so, we can set $\widehat{S}_1=S_1$ and $\widehat{S}_2=S_2$) and their 
spectral energies ($Y_{S_1}$ and $Y_{S_2}$) are depicted in Figure~\ref{Fig_4.5}.

\begin{figure}[h]
	\centering
	\includegraphics[width=6.5cm, height=6.5cm]{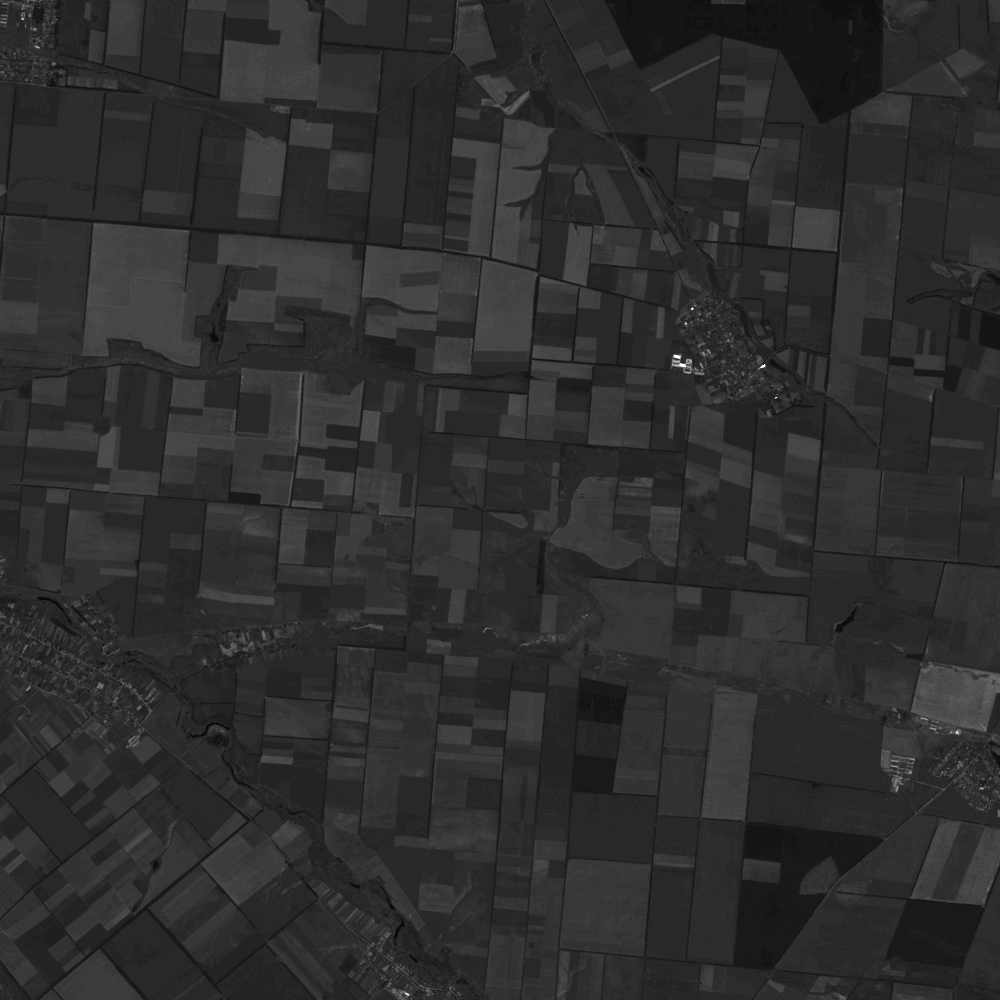}$\quad$
	\includegraphics[width=6.5cm, height=6.5cm]{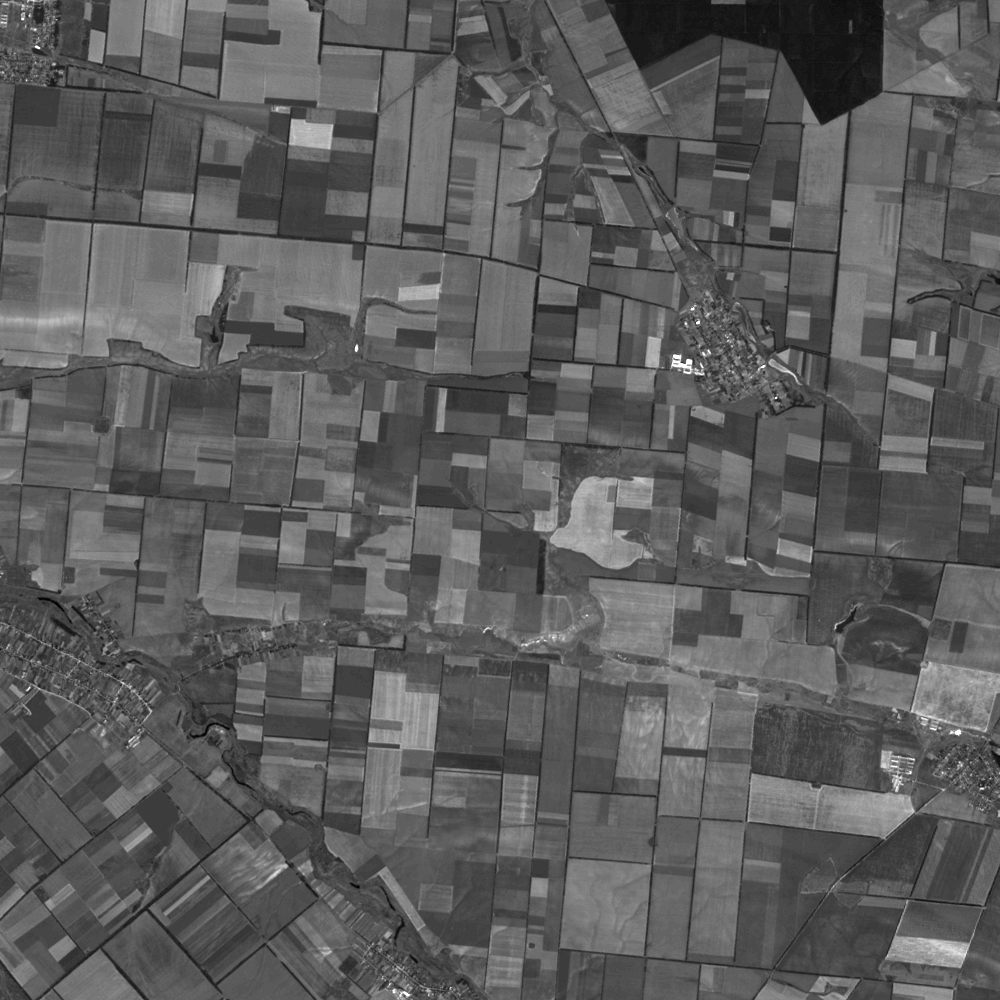}
	\caption{\label{Fig_4.5} The spectral energies of the Sentinel images $S_1$ and $S_2$ with resolution $20 m/pixel$. The real size of each image is $1000\times 1000$ pixels.}
\end{figure}

Setting $p_u$ in the form \eqref{1.6} with $h=0.1$, $a=0.01$, $\sigma=1$, and defining the function $v$ as a solution of the Neumann boundary value problem \eqref{4.7a}, we pass to the numerical solution of the IBVP \eqref{4.1a}--\eqref{4.1c}. To this end, we use an implicit discretization in time of equation \eqref{4.1a}, and after we apply the conjugate gradient method. As a result, we define a function $u^\ast(t,x)$ such that $u^\ast(t_1,x)=S_1(x)$, $\forall\,x\in G_H$,  and the peak signal-to-noise ratio between
$u^(t_2,\cdot)$ and $S_2(\cdot)$ is equal to $36.41$. It means that we can guarantee the equality $u^\ast(t_2,\cdot)\approx S_2(\cdot)$ at the high level of accuracy. Let $t_3='July, 18'$ and $t_4='August, 10'$ be some intermediate time point on the interval $[t_1,t_2]$.  In Figure~\ref{Fig_4.7} the screenshots $u^\ast(t_3,\cdot)$ and $u^\ast(t_4,\cdot)$ of the solution to the problem
\eqref{4.1a}--\eqref{4.1c} are shown. Thus, the function $u^\ast$ can be considered as an acceptable approximation to the evolution of the spectral energy of Sentinel images over the time interval  $[t_1,t_2]$.

\begin{figure}[h]
	\centering
	\includegraphics[width=6.5cm, height=6.5cm]{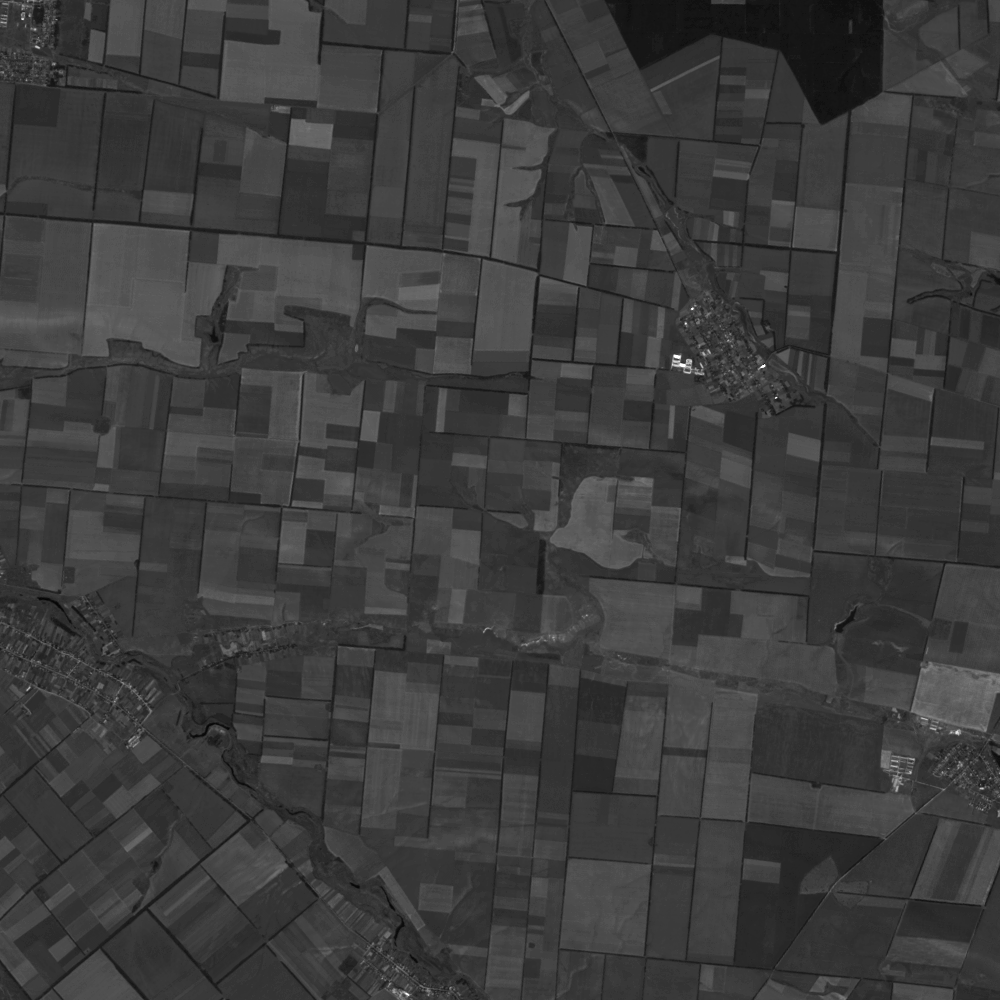}$\quad$
	\includegraphics[width=6.5cm, height=6.5cm]{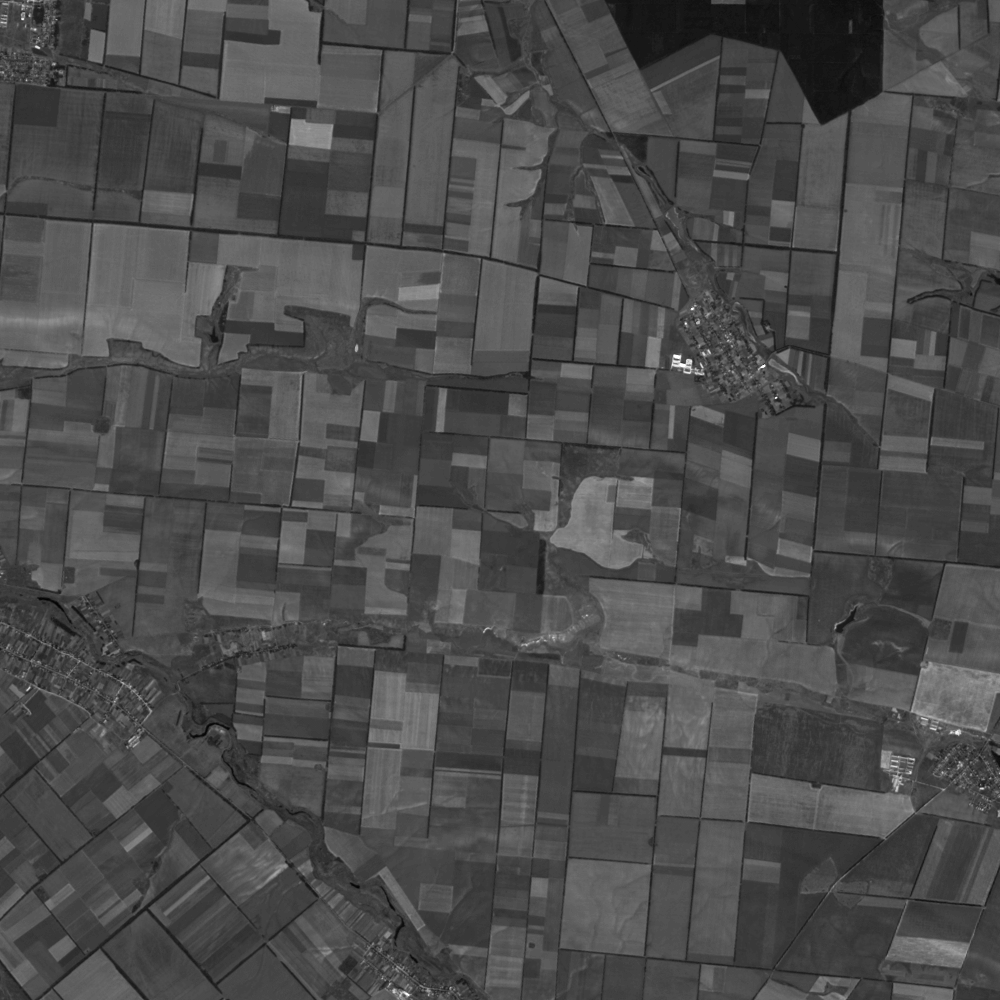}
	\caption{\label{Fig_4.7} The screenshots $u^(t_3,\cdot)$ and $u^\ast(t_4,\cdot)$ of the solution to the problem \eqref{4.1a}--\eqref{4.1c} taken at the time instances $t_3$ and $t_4$, respectively.}
\end{figure}

\section{Variational Statements of the Data Fusion Problems (A1)--(A3)}
\label{Sec 5}

Coming back to the principle cases (A1)--(A3), that have been described in Section~\ref{Sec 3}, we can suppose that a structural prototype $\widehat{S}_{t_M}:G_H\rightarrow \mathbb{R}^m$ from the given day $t_M$ is well defined. As it was emphasized in Section~\ref{Sec 3}, this prototype coincides either with one of the images $\left\{ \widehat{S}_i:G_H\rightarrow \mathbb{R}^m\right\}_{i=1}^N$ in cases (A1) and (A3), or it is defined using the solutions of the problem \eqref{4.1a}--\eqref{4.1c}, \eqref{4.7a} for each $j=1,\dots,m$ in the case (A2)-problem (see the rule \eqref{4.7}). For further convenience, we assume that $\widehat{S}_{t_M}:G_H\rightarrow \mathbb{R}^m$ is zero-extended outside of $\Omega$.

Let $j\in\left\{1,2,\dots,m\right\}$ be a fixed index value (the number of spectral channel). Let $q_j:\Omega\rightarrow \mathbb{R}$ be the texture index of the $j$-th band for the structural prototype $\widehat{S}_{t_M}:G_H\rightarrow \mathbb{R}^m$, i.e.
\begin{equation}
\label{5.1}
q_j(x):=1+g\left(\left|\left(\nabla G_\sigma\ast \widehat{S}_{t_M,j}\right) (x)\right|^2\right),\ \forall\,x\in \Omega,
\end{equation}
where
$g{:}[0,\infty)\rightarrow (0,\infty)$ is the edge-stopping function which we take in the form of the Cauchy law $
g(s)=\frac{a}{a+s}$ with $a>0$ small enough. 
Let $\eta\in (0,1)$ be a given threshold. Let $\theta_j=\left[\theta_{j,1}\atop \theta_{j,2}\right]\in L^\infty(\Omega,\mathbb{R}^2)$ be a vector field such that
\begin{equation*}
|\theta_j(x)|_{\mathbb{R}^2}\le 1\ \text{  and }\ \left(\theta_j(x),\nabla \widehat{S}_{t_M,j}(x)\right)_{\mathbb{R}^2}=| \nabla \widehat{S}_{t_M,j}(x)|_{\mathbb{R}^2}	\quad\text{a.e. in $\Omega$.}
\end{equation*}
As it was mentioned in Subsection~\ref{Subsec_1}, for each spectral channel $j\in\left\{1,\dots,m\right\}$ this vector-field can be defined by the rule 
$\theta_j(x)=\frac{\nabla U_j(t,x)}{|\nabla U_j(t,x)|}$ with $t>0$ small enough, where $U_j(t,x)$ is a solution the following initial-boundary value problem
\begin{align}
\label{5.2a}
&\frac{\partial U}{\partial t}=\mathrm{div}\,\left(\frac{\nabla U}{|\nabla U|+\e}\right),\quad   t\in(0,+\infty),\ x\in\Omega,\\
\label{5.2b}
&U(0,x)= \widehat{S}_{j,t_M}(x),\quad x\in\Omega,\\
\label{5.2c}
&\frac{\partial U(0,x)}{\partial\nu}=0,\quad t\in(0,+\infty),\ x\in\partial\Omega
\end{align}
with a relaxed version of the $1D$-Laplace operator in the principle part of \eqref{5.2a}. Here, $\e>0$ is a sufficiently small positive value.

Taking into account the definition of the Directional Total Variation (see \cite{Bung1}), we define a linear operator $R_{j,\eta}:\mathbb{R}^2\rightarrow \mathbb{R}^2$ as follows
\begin{equation}
\label{5.3}
R_{j,\eta}\nabla v:=\nabla v-\eta^2 \left(\theta_j,\nabla v\right)_{\mathbb{R}^2} \theta_j,\quad\forall\,v\in W^{1,1}(\Omega).
\end{equation}
It is clear that $R_{j,\eta}\nabla v$ reduces to $(1-\eta^2)\nabla v$ in those regions where the gradient $\nabla v$ is co-linear to $\theta_j$, and to $\nabla v$, where $\nabla v$ is orthogonal to $\theta_j$. It is important to emphasize that this operator does not enforce gradients in the direction $\theta_j$.

Let $\delta_{(x_i,y_j)}$ be the Dirac's delta at the point $(x_i,y_j)$. Then $\Pi_{S_L}=\sum_{(x_i,y_j)\in S_L} \delta_{(x_i,y_j)}$ stands for the Dirac's comb defined onto the sample grid $G_L$.

We are now in a position to give a precise meaning of the solutions to Problems (A1)--(A3). We say that:
\begin{enumerate}
	\item[(A1)]  A multi-band image $S^{rest}_{i^\ast}:G_H\to\mathbb{R}^m$, where $t_{i^\ast}=t_M$, is a solution of the Restoration Problem, if it 
	is given by the rule
	\begin{align}
	\label{5.5}
	S^{rest}_{i^\ast,j}(x)&=\left\{
	\begin{array}{ll}
	S_{i^\ast,j}(x) &,\ \forall\,x\in G_H\setminus D_{i^\ast},\\
	\beta_j u^0_j(x) &, \ \forall\,x\in G_H\cap D_{i^\ast},
	\end{array}
	\right.\ \forall\,j\in J_1,\\ 
	S^{rest}_{i^\ast,j}(x)&=\widehat{S}_{t_M,j}(x),\quad \forall\,x\in G_H,\ \forall\,j\in J_2.
	\end{align}
	Here, $\beta_j$ is the weight coefficient and we define it as follows
	\[
	\beta_j=\left[\int_{\Omega\setminus D_{i^\ast}} S_{i^\ast,j}(x)u^0_j(x)\,dx\right]/\left[\int_{\Omega\setminus D_{i^\ast}} |u^0_j(x)|^2\,dx\right]
	\]
	and $u^0_j$ is a solutions of the following constrained minimization problem
	\begin{equation}
	\label{5.5aa}
	\left(\mathcal{P}\right)\qquad\qquad
	\mathcal{F}_j\left(u^0_j\right)=\inf_{u\in\Xi_j}\mathcal{F}_j(u),
	\end{equation}
	where
	\begin{multline}
	\label{5.5ab}
	\mathcal{F}_j(u):=\int_{\Omega}\frac{1}{q_j(x)}|R_{j,\eta}\nabla u(x)|^{q_j(x)}\,dx
	+\frac{\mu}{2}\int_{\Omega}\left|\nabla u(x) -\nabla \widehat{S}_{t_M,j}(x)\right|^2\,dx\\
	+\frac{\gamma}{2}\int_{\Omega\setminus D_{i^\ast}}\left| u(x) -S_{t_M,j}(x)\right|^2\,dx
	+\frac{\vartheta}{2}\int_{\Omega}\Pi_{S_L}\Big( \left[\mathcal{K}\ast u-M_{t_M,j}\right]\Big)^2\,dx,
	\end{multline}
	$\Xi_j=\left\{u\in W^{1,q_j(\cdot)}(\Omega)\ :\ 0\le u(x)\le C_j\ \text{a.e. in }\ \Omega\right\}$ stands for the set of feasible solutions, $\mu>0$, $\gamma>0$, $\vartheta>0$ are some weight coefficients, and $W^{1,q_j(\cdot)}(\Omega)$ denotes the Sobolev space with variable exponent (for the details we refer to Appendix C). As for the constants $C_j$, their choice depends on the format of signed integer numbers in which the corresponding intensities $S_{i,j}(x)$ are represented. In particular, it can be $C_j=2^8-1$, $C_j=2^{16}-1$, and so on.

	\item[(A2)]  A multi-band image $S^{int}_{t_M}:G_H\to\mathbb{R}^m$, with $t_{i^\ast}<t_M<t_{i^\ast+1}$, is a solution of the Interpolation Problem, if it is given by the rule
	\begin{equation}
	\label{5.6}
	S^{int}_{t_M,j}(x)=\left\{
	\begin{array}{ll}
	\beta_j u^0_j(x) &,\ \forall\,j\in J_1,\\
	\widehat{S}_{t_M,j}(x) &, \ \forall\,j\in J_2,
	\end{array}
	\right.\ \forall\,x\in G_H,
	\end{equation}
	where 
	\begin{equation}
	\label{5.6a}
	\beta_j=\left[\int_{\Omega} \widehat{S}_{t_M,j}(x)u^0_j(x)\,dx\right]/\left[\int_{\Omega} |u^0_j(x)|^2\,dx\right]
	\end{equation}
	and $u^0_j$ is a solutions of the constrained minimization problem \eqref{5.5aa}--\eqref{5.5ab} with $D_{i^\ast}=\Omega$.
	
	\item[(A3)]  A multi-band image $S^{ext}_{t_M}:G_H\to\mathbb{R}^m$, with $t_N<t_M<T$, is a solution of the Extrapolation Problem, if it is given by the rule
	\begin{equation}
	\label{5.7}
	S^{ext}_{t_M,j}(x)=\left\{
	\begin{array}{ll}
	\beta_j u^0_j(x) &,\ \forall\,j\in J_1,\\
	\widehat{S}_{t_M,j}(x) &, \ \forall\,j\in J_2,
	\end{array}
	\right.\ \forall\,x\in G_H,
	\end{equation}
	where $u^0_j$ is a solutions of the constrained minimization problem \eqref{5.5aa}--\eqref{5.5ab}  with $D_{i^\ast}=\Omega$, and $\beta_j$ is defined as in \eqref{5.6a}.
\end{enumerate}  

Let us briefly discuss the relevance of the proposed minimization problem $\left(\mathcal{P}\right)$. We begin with the motivation to the choice of the energy functional in the form \eqref{5.5aa}--\eqref{5.5ab}. 

The first term in \eqref{5.5ab} can be considered as the regularization in the Sobolev-Orlicz space $W^{1,q_j(\cdot)}(\Omega)$ because, for each spectral channel, we have
\begin{equation}
\label{5.8.0}
(1-\eta^2)|\nabla u|\le |R_{j,\eta}\nabla u|\le |\nabla u|\quad\text{in $\Omega$}
\end{equation}
with a given threshold $\eta\in (0,1)$. Hence,
\begin{equation}
\label{5.8}
\ds\int_{\Omega}\frac{1}{q_j(x)}|R_{j,\eta}\nabla u(x)|^{q_j(x)}\,dx\ge
(1-\eta^2)^2\int_{\Omega}|\nabla u(x)|^{q_j(x)}\,dx,\quad\forall\,u\in W^{1,q_j(\cdot)}(\Omega)
\end{equation}
and, therefore, if
\begin{equation}
\label{5.8a}
u\in \Xi_j\subset W^{1,q_j(\cdot)}(\Omega)\cap L^\infty(\Omega)\quad\text{and}\quad
\mathcal{F}_j(u)<+\infty,
\end{equation}
then
\begin{align}
\notag
\|u\|^\alpha_{W^{1,q_j(\cdot)}}&=\left(\|u\|_{L^{q_j(\cdot)}(\Omega)}+\|\nabla u\|_{L^{q_j(\cdot)}(\Omega;\mathbb{R}^2)}\right)^\alpha\\
\notag
&\le C\left(\|u\|^\alpha_{L^{q_j(\cdot)}(\Omega)}+\|\nabla u\|^\alpha_{L^{q_j(\cdot)}(\Omega;\mathbb{R}^2)}\right)\\
\notag
&\stackrel{\text{by \eqref{A1.2}}}{\le}
C\left(\int_{\Omega} |u(x)|^{q_j(x)}\,dx+\int_{\Omega} |\nabla u(x)|^{q_j(x)}\,dx+2\right)\\
\notag
&\stackrel{\text{by \eqref{5.8a}}}{\le} C\left(|\Omega|C_j^{2}+\int_{\Omega} |\nabla u(x)|^{q_j(x)}\,dx+2\right)\\
\notag
&\stackrel{\text{by \eqref{5.8}}}{\le}
C\left(|\Omega|C_j^{2}+2+\frac{1}{(1-\eta^2)^2}\int_{\Omega}\frac{1}{q_j(x)}|R_{j,\eta}\nabla u(x)|^{q_j(x)}\,dx\right)\\
&\stackrel{\text{by \eqref{5.5ab}}}{\le}
C\left(|\Omega|C_j^{2}+2+\frac{1}{(1-\eta^2)^2}\mathcal{F}_j(u)\right)<+\infty. 
\label{5.9}
\end{align}

On the other side, this term plays the role of a spatial data fidelity. Indeed, what we are going to achieve in this interpolation problem, it is to preserve the
following property for the retrieved images at the Sentinel-2 resolution level: the geometry of each spectral channel of the retrieved image
has to be as close as possible to the geometry of the predicted structural prototype $\widehat{S}_{t_M}:G_H\rightarrow \mathbb{R}^m$  that we obtain either as a solution of the problem \eqref{4.1a}--\eqref{4.1c}, \eqref{4.7a}, or as a result of the iterative procedure \eqref{2.3}. Formally, it means that relations 
\begin{equation}
\label{3.8}
\left(\theta_j^\perp,\nabla u\right)=0\quad  \text{a.e. in $\Omega$}, \ \forall\, j\in J_1,
\end{equation}
have to be satisfied. Hence, the magnitude $\int_{\Omega}\Big|\left(\theta_j^{\perp},\nabla u\right)\Big|\,dx$
must be small enough for each spectral channel, where $\theta_j$ stands for the vector field of unit normals to the topographic map of the predicted band $\widehat{S}_{t_M,j}:G_H\rightarrow \mathbb{R}$. In order to achieve this property, we observe that the expression $R_{j,\eta}\nabla u$ can be reduced to $(1-\eta^2)\nabla u$ in those places of  $\Omega$ where $\nabla u$ is co-linear to the unit normal $\theta_j$, and to $\nabla u$ if $\nabla u$ is orthogonal to $\theta_j$. 

Thus, gradients of the   intensities $u$ that are aligned/co-linear to $\theta_j$ are favored as long as $|\theta_j|>0$. Moreover, this property is enforced by the special choice of the exponent $q_j(x)$.
Since $q_j(x)\approx 1$ in places in $\Omega$ where edges or discontinuities are present in the predicted band $\widehat{S}_{t_M,j}$, and $q_j(x)\approx 2$ in places where $\widehat{S}_{t_M,j}$ is smooth or contains homogeneous features, the main benefit of the energy functional \eqref{5.5ab} is the manner in which it accommo\-dates the local image information. For the places where the gradient of $\widehat{S}_{t_M,j}$ is sufficiently large (i.e. likely edges), we deal with the so-called directional TV-based diffusion \cite{Bung1,Bung2}, whereas in the places where the gradient of $\widehat{S}_{t_M,j}$ is close to zero (i.e. homogeneous regions), the model becomes isotropic. Specifically, the type of anisotropy at these ambiguous regions varies according to the strength of the gradient.  Apparently,  the idea to involve the norm of $W^{1,q_j(\cdot)}(\Omega)$ with a variable exponent $q_j(x)$ was firstly proposed in \cite{Blom2} in order to reduce the staircasing effect in the TV image restoration problem.

As for the second term in \eqref{5.5ab}, it reflects the fact that the topographic map of the retrieved image should be as close as possible to the topographic map of predicted structural prototype $\widehat{S}_{t_M}:G_H\rightarrow \mathbb{R}^m$. We interpret this closedness in its simplified form, namely, in the sense of $L^2$-norm of the difference of the corresponding gradients.

It remains to say a few words about the last term in \eqref{5.5ab}.  Basically, this term represents an $L^2$-distortion between a $j$-th spectral channel in the MODIS image $M=\left[M_1,M_2,\dots,M_6\right]:G_L\to\mathbb{R}^6$ and the corresponding channel of the retrieved image $u^0_j$ which is resampled to the grid of low resolution $G_L$. 

%%%%%%%%%%%%%%%%%%%%%%%%%%%%%%%%%%%%%%%%%%%%%%%%%%%%%%%

\section{Existence Result and Optimality Conditions for Constrained Minimization Problem  $\left(\mathcal{P}_k\right)$}
\label{Sec 7}

Our main intention in this section is to show that, for each $j\in J_1$,  constrained minimization problem \eqref{5.5aa}--\eqref{5.5ab} is consistent and admits at least one solution. Because of the specific form of the energy functional $\mathcal{F}_j(u)$, the minimization problem \eqref{5.5aa}--\eqref{5.5ab} is rather challenging and we refer to \cite{Blom1,Blom2,Bung1,Chen} for some specific details.

Following in many aspects the recent studies \cite{CK2,K3} (see also \cite{DAKOMA2022,DAKOMAUV2022,DAKOKUMA2022,HK}), we can give the following existence result.
\begin{theorem}
	\label{Th 7.1}
	Let $\widehat{S}_{t_M}:G_H\rightarrow \mathbb{R}^m$ be a given structural prototype for unknown image ${S}_{t_M}$ from Sentinel-2. 
	Then for any given $j\in J_1$, $\mu>0$, $\gamma>0$, $\vartheta>0$, and $\eta\in (0,1)$,  the minimization problem \eqref{5.5aa}--\eqref{5.5ab} admits a unique solution $u_j^0\in \Xi_j$.
\end{theorem}

In order to derive some optimality conditions to the problem \eqref{5.5aa}--\eqref{5.5ab} 
and charac\-terize its solution $u_j^0\in W^{1,q_j(\cdot)}(\Omega)$, we show that the cost functional  $\mathcal{F}_j:\Xi_j\rightarrow\mathbb{R}$ is G\^{a}teaux differentiable. To this end, we note that, for arbitrary $v\in W^{1,q_j(\cdot)}(\Omega)$, the following assertion 
\begin{multline*}
\frac{|R_{j,\eta}\nabla u^0_j(x)+tR_\eta\nabla v(x)|^{q_j(x)}-|R_{j,\eta}\nabla u^0_j(x)|^{q_j(x)}}{q_j(x) t}\\
\rightarrow
\left(|R_{j,\eta}\nabla u^0_j(x)|^{q_j(x)-2}R_{j,\eta}\nabla u^0_j(x),R_{j,\eta}\nabla v(x)\right)\ \text{ as }\ t\to 0
\end{multline*}
holds almost everywhere in $\Omega$. Indeed, by convexity,
\[
|\xi|^{q_j(x)}-|\eta|^{q_j(x)}\le 2q_j(x)\left(|\xi|^{q_j(x)-1}+|\eta|^{q_j(x)-1}\right)|\xi-\eta|,
\]
it follows that
\begin{multline}
\left|\frac{|R_{j,\eta}\nabla u^0_j(x)+tR_{j,\eta}\nabla v(x)|^{q_j(x)}-|R_{j,\eta}\nabla u^0_j(x)|^{q_j(x)}}{q_j(x) t}\right|\\
\le 2\left(|R_{j,\eta}\nabla u^0_j(x)+tR_{j,\eta}\nabla v(x)|^{q_j(x)-1}+|R_{j,\eta}\nabla u^0_j(x)|^{q_j(x)-1}\right)|R_{j,\eta}\nabla v(x)|\\
\le \mathrm{const}\, \left(|R_{j,\eta}\nabla u^0_j(x)|^{q_j(x)-1}+|R_{j,\eta}\nabla v(x)|^{q_j(x)-1}\right)|R_{j,\eta}\nabla v(x)|.
\label{7.1.3}
\end{multline}
Taking into account that
\begin{multline*}
\int_\Omega |R_{j,\eta}\nabla u^0_j(x)|^{q_j(x)-1}|R_{j,\eta}\nabla v(x)|\,dx\\
\stackrel{\text{by \eqref{AKog_1.2.1}}}{\le}
2\|R_{j,\eta}\nabla u^0_j(x)|^{q_j(x)-1}\|_{L^{(q_j)^\prime(\cdot)}(\Omega)}\|R_{j,\eta}\nabla v(x)|\|_{L^{q_j(\cdot)}(\Omega)}\\
\le 2\|\nabla u^0_j(x)|^{q_j(x)-1}\|_{L^{(q_j)^\prime(\cdot)}(\Omega,\mathbb{R}^2)}\|\nabla v(x)\|_{L^{q_j(\cdot)}(\Omega,\mathbb{R}^2)},
\end{multline*}
and
\begin{align*}
\int_\Omega |R_{j,\eta}\nabla v(x)|^{q_j(x)}\,dx\stackrel{\text{by \eqref{5.8.0}}}{\le}
\int_\Omega |\nabla v(x)|^{q_j(x)}\,dx& \le \|\nabla v\|^2_{L^{q_j(\cdot)}(\Omega,\mathbb{R}^2)}+1,
\end{align*}
we see that the right hand side of inequality \eqref{7.1.3} is an $L^1(\Omega)$ function. Therefore,
\begin{multline*}
\int_\Omega\frac{|R_{j,\eta}\nabla u^0_j(x)+tR_{j,\eta}\nabla v(x)|^{q_j(x)}-|R_{j,\eta}\nabla u^0_j(x)|^{q_j(x)}}{q_j(x) t}\,dx\\
\rightarrow
\int_\Omega\left(|R_{j,\eta}\nabla u^0_j(x)|^{q_j(x)-2}R_{j,\eta}\nabla u^0_j(x),R_{j,\eta}\nabla v(x)\right)\,dx\ \text{ as }\ t\to 0
\end{multline*}
by the Lebesgue dominated convergence theorem.

Since the cost functional $\mathcal{F}_j:\Xi_j\rightarrow\mathbb{R}$ can be cast in the form
\[
P_k(u)=\mathcal{A}_j(u)+\mu \mathcal{B}_j(u) +\gamma\mathcal{C}_j(u) +\vartheta \mathcal{D}_j(u),
\]
where
\begin{align*}
\mathcal{A}_j(u)&=\int_{\Omega}\frac{1}{q_j(x)}|R_{j,\eta}\nabla u(x)|^{q_j(x)}\,dx,\\
\mathcal{B}_j(u)&=\frac{1}{2}\int_{\Omega}\left|\nabla u(x) -\nabla \widehat{S}_{t_M,j}(x)\right|^2\,dx,\\
\mathcal{C}_j(u)&=\frac{1}{2}\int_{\Omega\setminus D_{i^\ast}}\left|u(x) -S_{t_M,j}(x)\right|^2\,dx,\\
\mathcal{D}_j(u)&=\frac{1}{2}\int_{\Omega}\Pi_{S_L}\Big( \left[\mathcal{K}\ast u-M_{t_M,j}\right]\Big)^2\,dx,
\end{align*}
we deduce that
\begin{equation}
\mathcal{A}_j^\prime(u^0_j)[v]=
\int_{\Omega}\left(|R_{j,\eta}\nabla u^0_j(x)|^{q_j(x)-2}R_{j,\eta}\nabla u^0_j(x),R_{j,\eta}\nabla v(x)\right)_{\mathbb{R}^2}\,dx,
\end{equation}
for each $v\in W^{1,q_j(\cdot)}(\Omega)$.

As for the rest terms $\mu \mathcal{B}_j(u)$, $\gamma\mathcal{C}_j(u)$, and $\lambda \mathcal{D}_j(u)$ in the cost functional  $\mathcal{F}_j:\Xi_j\rightarrow\mathbb{R}$, utilizing the similar arguments, we have the following representation for their G\^{a}teaux derivatives.
\begin{proposition}
	\label{Prop 7.4}
	For a given MODIS  image $M:G_L\to\mathbb{R}^6$,  the functionals $\mathcal{B}_j,\mathcal{C}_j, \mathcal{D}_j:L^2(\Omega)\rightarrow \mathbb{R}$ are convex and G\^{a}teaux  differentiable in $L^2(\Omega)$ with
	\begin{align}
	\label{7.5b}
	\mathcal{B}_k^\prime(u^0_j)[v]&=
	\int_{\Omega}\left(\nabla u^0_j(x) -\nabla \widehat{S}_{t_M,j}(x),\nabla v(x)\right)\,dx,\\
	\mathcal{C}_k^\prime(u^0_j)[v]&=
	\int_{\Omega\setminus D_{i^\ast}}\left(u^0_j(x) - S_{t_M,j}(x)\right) v(x)\,dx,\\
	\notag
	\mathcal{D}_k^\prime(u^0_j)[v]&=
	\int_{\Omega} \Pi_L\left( \left[\mathcal{K}\ast u^0_j\right]-M_{t_M,j}\right)\left[\mathcal{K}\ast v\right]\,dx\\
	&=
	\int_{\Omega} \Pi_L \left[\mathcal{K}^\ast\ast \left( \left[\mathcal{K}\ast u^0_j\right]-M_{t_M,j}\right)\right] v\,dx,	
	\label{7.5c}
	\end{align}
	for all $v\in W^{1,q_j(\cdot)}(\Omega)$.
\end{proposition}

Thus, in order to derive some optimality conditions for the minimizer  $u^0_j\in W^{1,q_j(\cdot)}(\Omega)$ to the problem $\inf\limits_{u\in \Xi_j} \mathcal{F}_j(u)$, we note that $\Xi_j$ is a nonempty convex subset of $W^{1,q_j(\cdot)}(\Omega)\cap L^\infty(\Omega)$ and the objective functional $\mathcal{F}_j: \Xi_j\rightarrow\mathbb{R}$ is strictly convex. Hence, the well known classical result (see \cite[Theorem~1.1.3]{Lions71}) leads us to the following conclusion.

\begin{theorem}
	\label{Th 7.9}
	Let $\widehat{S}_{t_M}:G_H\rightarrow \mathbb{R}^m$ be a given structural prototype for unknown image ${S}_{t_M}$ from Sentinel-2. Let $M:G_L\to\mathbb{R}^6$ be a given  MODIS  image. Let $q_j$ stands for the texture index of the $j$-th band for the predicted structural prototype $\widehat{S}_{t_M}$ (see \eqref{5.1}).
	Then the unique minimizer $u^0_j\in \Xi_j$ to the minimization problem $\inf\limits_{u\in \Xi_j}\mathcal{F}_j(u)$ is characterized by the following variational inequality
	\begin{multline}
	\label{7.21}
	\int_{\Omega}\left(\Big|R_{j,\eta}\nabla u^0_j(x)\Big|^{q_j(x)-2}R_{j,\eta}\nabla u^0_j(x),R_{j,\eta}\nabla v(x)-R_{j,\eta}\nabla u^0_j(x)\right)\,dx\\
	+\mu
	\int_{\Omega}\left(\nabla u^0_j(x) -\nabla \widehat{S}_{t_M,j}(x),\nabla v(x)\right)\,dx\\
	+\gamma
	\int_{\Omega\setminus D_{i^\ast}}\left(u^0_j(x) - S_{t_M,j}(x)\right) v(x)\,dx\\
	+\vartheta\int_{\Omega} \Pi_L \left[\mathcal{K}^\ast\ast \left( \left[\mathcal{K}\ast u^0_j\right]-M_{t_M,j}\right)\right] v\,dx\ge 0,\	\quad\forall\,v\in \Xi_j.
	\end{multline}
\end{theorem}

\begin{remark}
	In practical implementation, it is reasonable to define an optimal solution  $u^0_j\in \Xi_j$  using a 'gradient descent' strategy.  Indeed, observing that
	\begin{multline*}
	\int_{\Omega}\left(\Big|R_{j,\eta}\nabla u^0_j(x)\Big|^{q_j(x)-2}R_{j,\eta}\nabla u^0_j(x),\nabla v(x)\right)\,dx\\
	=-\int_{\Omega}\div\left(\Big|R_{j,\eta}\nabla u^0_j(x)\Big|^{q_j(x)-2}R_{j,\eta}\nabla u^0_j(x)\right)  v\,dx\\
	+
	\int_{\partial\Omega} \left(\Big|R_{j,\eta}\nabla I_{k,i}(x)\Big|^{q_j(x)-2}R_{j,\eta}\nabla u^0_j(x),\nu\right) v\,d\mathcal{H}^1,
	\end{multline*}
	and
	\begin{multline*}
	\int_{\Omega}\left(\Big|R_{j,\eta}\nabla u^0_j(x)\Big|^{q_j(x)-2}R_{j,\eta}\nabla u^0_j(x),\left(\theta,\nabla v\right) \theta\right)\,dx\\
	=-\int_{\Omega}\div\left(\left(\Big|R_{j,\eta}\nabla u^0_j(x)\Big|^{q_j(x)-2}R_{j,\eta}\nabla u^0_j(x),\theta\right)\theta\right)  v\,dx\\
	+
	\int_{\partial\Omega} \left(\Big|R_{j,\eta}\nabla u^0_j(x)\Big|^{q_j(x)-2}R_{j,\eta}\nabla u^0_j(x),\theta\right)\left(\theta,\nu\right) v\,d\mathcal{H}^1,
	\end{multline*}
	we see that
	\begin{multline*}
	\int_{\Omega}\left(\Big|R_{j,\eta}\nabla u^0_j(x)\Big|^{q_j(x)-2}R_{j,\eta}\nabla u^0_j(x),R_{j,\eta}\nabla v(x)\right)\,dx\\
	=\int_{\Omega}\left(\Big|R_{j,\eta}\nabla u^0_j(x)\Big|^{q_j(x)-2}R_{j,\eta}\nabla u^0_j(x),\nabla v-\eta^2 \left(\theta,\nabla v\right) \theta\right)\,dx\\
	=-\int_{\Omega}\div\left(\Big|R_{j,\eta}\nabla u^0_j(x)\Big|^{q_j(x)-2}R_{j,\eta}\nabla u^0_j(x)\right)  v\,dx\\
	+\eta^2 \int_{\Omega}\div\left(\left(\Big|R_{j,\eta}\nabla u^0_j(x)\Big|^{q_j(x)-2}R_{j,\eta}\nabla u^0_j(x),\theta\right)\theta\right)  v\,dx
	\end{multline*}
	provided
	\[
	\left(\Big|R_{j,\eta}\nabla u^0_j(x)\Big|^{q_j(x)-2}R_{j,\eta}\nabla u^0_j,\nu\right)=0\quad\text{on}\ \partial\Omega.
	\]
	
	Thus, following the standard procedure and starting from the initial image $\widehat{S}_{t_M,j}$, we can pass to the following initial value problem for the quasi-linear parabolic equations with Nuemann boundary conditions
	\begin{align}
	\notag
	\frac{\partial u^0_j}{\partial t}-& \div\left(|R_{j,\eta}\nabla u^0_j(x)|^{q_j(x)-2}R_{j,\eta}\nabla u^0_j(x)\right)\\ 
	\notag
	&=-\eta^2 \div\left(\left(\Big|R_{j,\eta}\nabla u^0_j(x)\Big|^{q_j(x)-2}R_{j,\eta}\nabla u^0_j(x),\theta\right)\theta\right)\\
	\notag
	&\quad+\mu \div\left(\nabla u^0_j(x) -\nabla \widehat{S}_{t_M,j}(x)\right)-\gamma\left(u^0_j(x)-{S}_{t_M,j}(x)\right)\\
	\label{7.1}
	&\quad-\vartheta \Pi_L \left[\mathcal{K}^\ast\ast \left( \left[\mathcal{K}\ast u^0_j\right]-M_{t_M,j}\right)\right],\\
	\label{7.3}
	&\left(\Big|R_{j,\eta}\nabla u^0_j(x)\Big|^{q_j(x)-2}R_{j,\eta}\nabla u^0_j,\nu\right)=0\quad\text{on}\ \partial\Omega,\\
	\label{7.4}
	0&\le  u^0_j(x)\le C_j\ \text{a.a. in}\ \Omega,\\
	\label{7.5}
	&u^0_j(0,x)=\widehat{S}_{t_M,j}(x),\quad\forall\,x\in\Omega. 
	\end{align}
In principle, instead of the initial condition \eqref{7.5} we may  consider other image that can be generated from $\widehat{S}_{t_M,j}$ and the bicubic interpolation of the MODIS band $M_{t_M,j}$ onto the entire domain $\Omega$. For instance, it can be one of well-known simple data fusion methods (for the details, we refer to \cite{Rani}).
\end{remark}

%%%%%%%%%%%%%%%%%%%%%%%%%%%%%%%%%%%%%%%%%%%%%%%%%%%%%%%%%%%%%%%%%%%%%

\section{Numerical Experiments}
\label{Sec 8}

In order to illustrate the proposed approach for the restoration of satellite multi-spectral images we have used a series of Sentinel-2  images ($725\times 600$ in pixels) over the South Dakota area (USA) with resolution $10 m/pixel$ (see Fig.~\ref{Fig_1} that were captured at different time instances in period from July 10 to July 15, 2021, when the global biophysical processes are rapid enough. So, it was a period of an active vegetation growing, and what is more important, each of these images have been captured at a cloud-free day. As follows from Fig.~\ref{Fig_1}, this region represents a typical agricultural area with medium sides fields of various shapes. Since each of these images contains four bands --- $B_4$ (red), $B_3$ (green), $B_2$ (blue), and $B_{8a}$ (near infrared), we denote these images as $S_1, S_2, S_3:G_H\to\mathbb{R}^4$, respectively. We also have a cloud-free MODIS image $M:G_L\to\mathbb{R}^6$ ($29\times 24$ in pixels) from 2021/08/13 with resolution $250 m/pixel$ (see Fig.~\ref{Fig_2}).

\begin{figure}[ht!]
	\centering
	\includegraphics[width=6.5cm, height=6cm, scale=4.5]{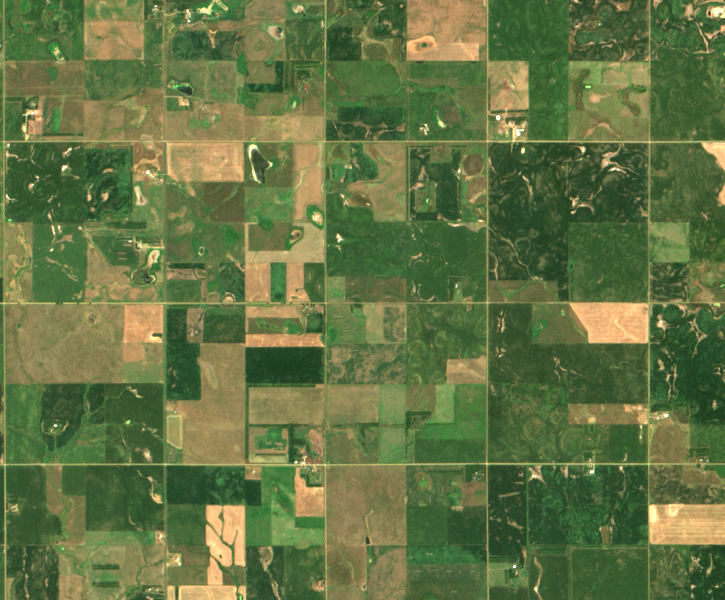}\  
	\includegraphics[width=6.5cm, height=6cm, scale=4.5]{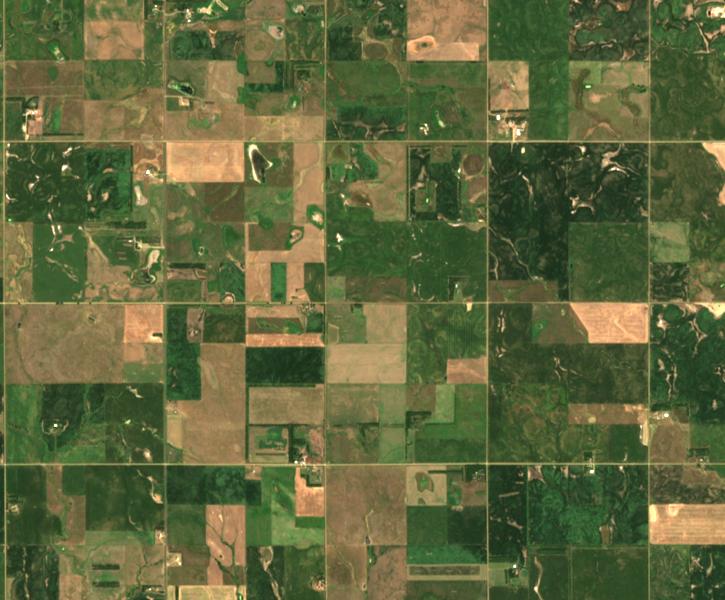}\  
	\includegraphics[width=6.5cm, height=6cm, scale=4.5]{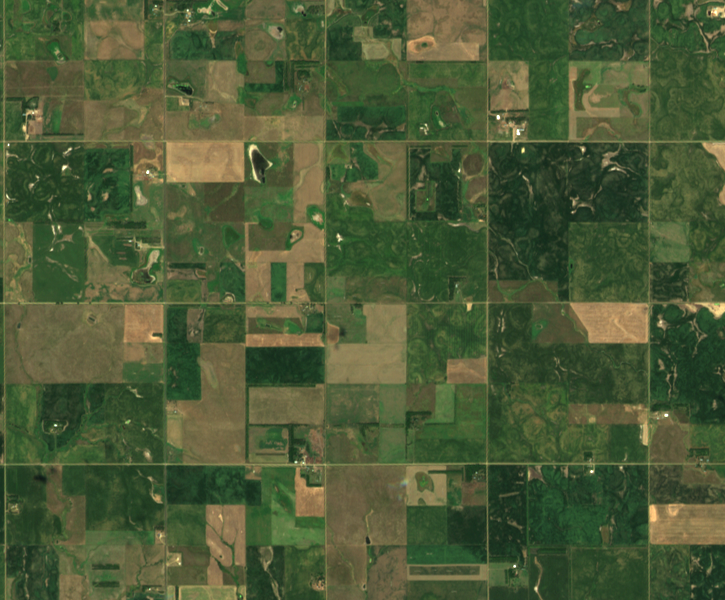}
	\caption{\label{Fig_1}Images from Sentinel-2. Date of generation: (left up)~2021/08/10, (right up)~2021/08/13, , (center down)~2021/08/15}
\end{figure}
\begin{figure}[ht!]
	\centering
	\includegraphics[width=1cm, height=1cm, scale=4.5]{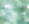} 
	\caption{\label{Fig_2}MODIS image  with the date of generation 2021/08/13}
\end{figure}

To emulate the interpolation problem (A2), we assume that the image $S_2$ is cloud-corrupted. With that in mind, we create some artificial clouds on it (see the left picture in Fig.~\ref{Fig_3}) and consider a new image $\mathcal{S}_2:G_ H\setminus D\to \mathbb{R}^4$ as an image with damage zone $D$ in all bands. So, the problem consists in generation of a new multi-band optical image $\mathcal{S}^{int}_{2}:G_H\to\mathbb{R}^4$ at the Sentinel-level of resolution using result of the fusion of cloud-free MODIS image $M:G_L\to\mathbb{R}^6$ with the predicted structural prototype $\widehat{\mathcal{S}}_{2}:G_H\rightarrow \mathbb{R}^m$ from the given day 2021/08/13. 
\begin{figure}[ht!]
	\centering
	\includegraphics[width=6.5cm, height=6cm, scale=4.5]{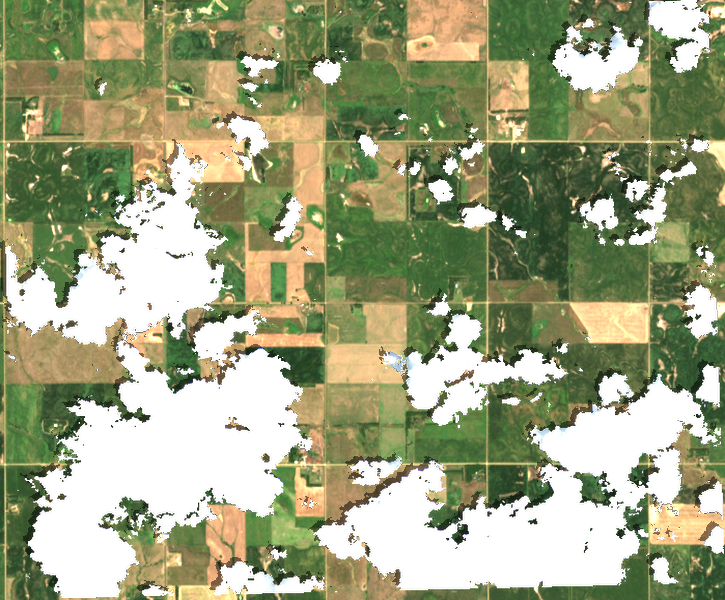}\quad
	\includegraphics[width=6.5cm, height=6cm, scale=4.5]{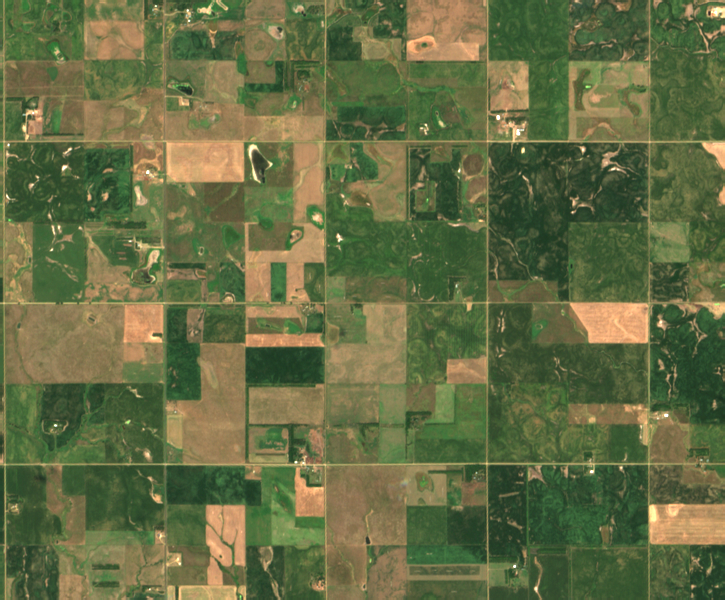}
	\caption{\label{Fig_3} (left)~Image with damage, (right) Its structural prototype}
\end{figure}

At the first step, following the procedure described in Section~4 and utilizing for that the two cloud-free Sentinel images $S_1$ and $S_3$, we create for each band its daily prediction of the topographical map from the day 2021/08/13 (see the right picture in Fig.~\ref{Fig_3}). After that, we realize the fusion procedure of this predicted image with the cloud-free MODIS image $M:G_L\to\mathbb{R}^6$ of the same territory.  In all numerical simulations, we set $h=0.1$, $a=0.01$, $\sigma=1$, $\e=0.001$, $\eta=0.95$, $\mu=2.5$, $\vartheta=1$, $\gamma=0$. As a result, a new image from the date 2021/08/13 at the Sentinel-level of resolution has been generated and it is depicted in Fig.~\ref{Fig_4}. 

\begin{figure}[ht!]
	\centering
	\includegraphics[width=6.5cm, height=6cm, scale=4.5]{13082021_O.png}\quad
	\includegraphics[width=6.5cm, height=6cm, scale=4.5]{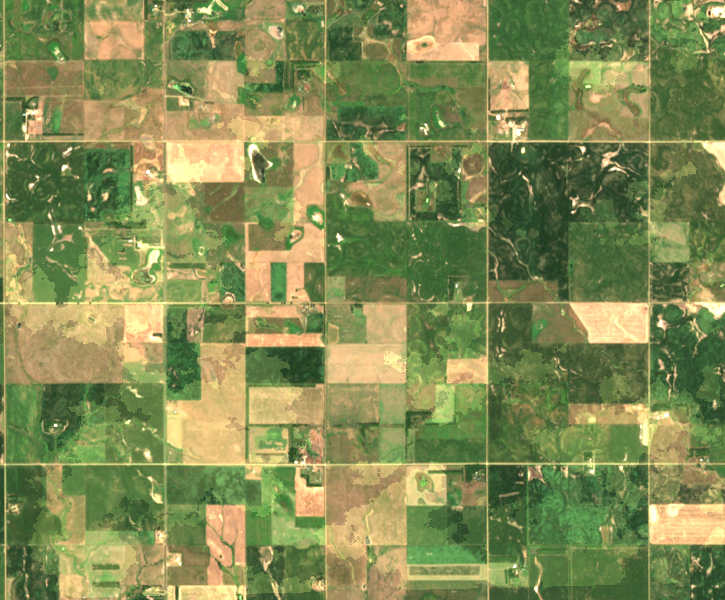}
	\caption{\label{Fig_4} (left)~Original image from Sentinel from 2021/08/13, (right) Result of its restoration}
\end{figure}

To evaluate the given interpolation result, we make use of the following validation metrics: 
 
\begin{center}
	\begin{tabular}{ l | l   }
		\hline
		\textbf{RMSE} &  $\mathrm{RMSE}\left(S_2,\mathcal{S}^{int}_{2}\right)=\frac{1}{|G_H|}\sum_{x\in G_H} \left(S_2(x)-\mathcal{S}^{int}_{2}(x)\right)^2$ \\ \hline
		\textbf{Corr} &  $\mathrm{Corr}\left(S_2,\mathcal{S}^{int}_{2}\right)=\frac{\mathrm{cov}\,\left(S_2,\mathcal{S}^{int}_{2}\right)}{\sqrt{\mathrm{cov}\,\left(S_2,S_2\right) \mathrm{cov}\,\left(\mathcal{S}^{int}_{2},\mathcal{S}^{int}_{2}\right)}}$ \\ \hline
		\textbf{CorrLaplace} &  $\mathrm{CorrLaplace}\left(S_2,\mathcal{S}^{int}_{2}\right)=\mathrm{Corr}\left(\Delta S_2, \Delta \mathcal{S}^{int}_{2}\right)$ \\
		\hline
		\textbf{SSIM} &  The structural similarity index combining  local image structure,\\
		 &  luminance, and contrast into a single local quality score \\
		\hline
		\textbf{HaarPSI} &  The Haar wavelet-based perceptual similarity index \\
		 & that aims to correctly assess the perceptual similarity\\
		   & between two images  with respect to a human viewer \\
		\hline
	\end{tabular}
\end{center}
\vspace{2ex}

\noindent
where the above abbreviations stand for: \textbf{RMSE} (The Root-Mean-Square Deviation), \textbf{Corr} (The Correlation Coefficient), \textbf{CorrLaplace} (The Correlation of Laplace Feature).  

Table~\ref{Table_1} contains the comparison results of the original image $S_2$ with its restored variant $\mathcal{S}^{int}_{2}$ for each band separately, whereas in Table~\ref{Table_2} we present the similarity results for the normalized difference vegetation indices (NDVI) which can be calculated from these images as follows:
\[
NDVI(S_2)=\frac{B_{8a}(S_2)-B_4(S_2)}{B_{8a}(S_2)+B_4(S_2)},\qquad
NDVI(\mathcal{S}^{int}_2)=\frac{B_{8a}(\mathcal{S}^{int}_2)-B_4(\mathcal{S}^{int}_2)}{B_{8a}(\mathcal{S}^{int}_2)+B_4(\mathcal{S}^{int}_2)}.
\]

\begin{table}[ht]
	\caption{ Similarity characteristics between $S_2$ and its restored variant $\mathcal{S}^{int}_{2}$}
	\centering
	\begin{tabular}{ l | c | c | c | c   }
		\hline
		Metrics & $B_2$ & $B_3$ & $B_4$ & $B_{8a}$ \\ \hline
		\textbf{RMSE} & 1669.9299 & 6339.3894 & 5587.9002 & 143613.5512	  \\ \hline
		\textbf{Corr} &  0.9720	& 0.9272 & 0.9769 & 0.9220	 \\ \hline
		\textbf{CorrLaplace} & 0.9490 &	0.8623 & 0.9229 & 0.3107  \\
		\hline
		\textbf{SSIM} & 0.9428 & 0.9094	& 0.9313 & 0.8665   \\
		\hline
		\textbf{HaarPSI} & 0.8008 & 0.6994 & 0.7907 & 0.6498  \\
		\hline
	\end{tabular}
\label{Table_1}
\end{table} 

\begin{table}[ht]
	\caption{ Similarity characteristics between NDVIs for $S_2$ and its restored variant $\mathcal{S}^{int}_{2}$}
	\centering
	\begin{tabular}{ l | c }
		Metrics &  Value \\ \hline
		$\mathrm{RMSE}\left(\mathrm{NDVI}(S_2),\mathrm{NDVI}(\mathcal{S}^{int}_{2})\right)$ & 0.0002  \\ \hline
		$\mathrm{SSIM}\left(\mathrm{NDVI}(S_2),\mathrm{NDVI}(\mathcal{S}^{int}_{2})\right)$  & 0.9999 \\
		\hline
		$\mathrm{HaarPSI}\left(\mathrm{NDVI}(S_2),\mathrm{NDVI}(\mathcal{S}^{int}_{2})\right)$  & 0.9999  \\
		\hline
	\end{tabular}
	\label{Table_2}
\end{table} 

As for the choice of numerical schemes for the problems \eqref{4.7a} and \eqref{7.1}--\eqref{7.5}, their consistency and substantiation, and also other scenario for simulations including the restoration and prediction problems, these issues will be a subject of a forthcoming paper. We will show that the proposed approach is appealing for automated processing of large data sets. 

\section{Conclusion}

We propose a novel  model for the restoration of satellite multi-spectral images. This model is based on  the solutions of special variational problems with nonstandard growth objective functional. 
Because of the risk of information loss in optical images, we do not impose any information about such images inside the damage region, but instead we assume that the texture of these images can be predicted through a number of past cloud-free images of the same region from the time series. So, the characteristic feature of variational problems, which we formulate for each spectral channel separately, is the structure of their objective functionals. On the one hand, we involve into consideration the energy functionals with the nonstandard growth  $p(x)$, where the variable exponent $p(x)$ is unknown a priori and it directly depends on the texture of an image that we are going to restore. On the other hand, the texture of an image $\vec{u}$, we are going to restore, can have rather rich structure in the damage region $D$. In order to identify it, we push forward the following hypothesis: the geometry of each spectral channels of a cloud corrupted image in the damage region is topologically close to the geometry of the total spectral energy that can be predicted with some accuracy by a number of past cloud-free images of the same region. As a result, we impose this requirement in each objective functional in the form of a special fidelity term. In order to study the consistency of the proposed collection of non-convex minimization problems, we develop a special technique and supply this approach by the rigorous mathematical substantiation.

\appendix
\section{On Orlicz Spaces}
\label{SubSec 1.1}
Let $p(\cdot)$ be a measurable exponent function on $\Omega$ such that
$1<\alpha\le p(x)\le\beta<\infty$ a.e. in $\Omega$, where $\alpha$ and $\beta$ are given constants. Let $p^\prime(\cdot)=\frac{p(\cdot)}{p(\cdot)-1}$ be the corresponding conjugate exponent. It is clear that
\[
1\le \underbrace{\frac{\beta}{\beta-1}}_{\beta^\prime}\le p^\prime(x)\le\underbrace{\frac{\alpha}{\alpha-1}}_{\alpha^\prime}\ \ a.e.\  in \ \Omega,
\]
where $\beta^\prime$ and $\alpha^\prime$ stand for the
conjugates of constant exponents. Denote by $L^{p(\cdot)}(\Omega)$ the set of all measurable functions $f(x)$ on $\Omega$ such that $\int_\Omega |f(x)|^{p(x)}\,dx<\infty$. Then $L^{p(\cdot)}(\Omega)$
is a reflexive separable Banach space with respect to the Luxemburg norm (see  \cite{Cruz, DHHR} for the details)
\begin{equation}
\label{A0.8}
\|f\|_{L^{p(\cdot)}(\Omega)}=\inf\left\{\lambda>0\ :\ \rho_p(\lambda^{-1}f)\le 1\right\},\
\end{equation}
where $\rho_p(f):=\int_\Omega |f(x)|^{p(x)}\,dx$.

It is well-known that $L^{p(\cdot)}(\Omega)$ is reflexive provided $\alpha>1$, and its dual is $L^{p^\prime(\cdot)}(\Omega)$, that is, any continuous functional $F=F(f)$ on $L^{p(\cdot)}(\Omega)$ has the form (see \cite[Lemma~13.2]{Zh2011})
\[
F(f)=\int_{\Omega} f g \,dx,\quad \ with \ \ g\in L^{p^\prime(\cdot)}(\Omega).
\] 
As for the infimum in (\ref{A0.8}), we have the following result.
\begin{proposition}
	\label{Kog_Prop A0.8}
	The infimum in (\ref{A0.8}) is attained if $\rho_p(f)>0$. Moreover
	\begin{equation}
	\label{AKog_1.1}
	if \ \lambda_\ast:=\|f\|_{L^{p(\cdot)}(\Omega)}>0,\ \ then \ \ \rho_p(\lambda_\ast^{-1}f)=1.
	\end{equation}
\end{proposition}

Taking this result and condition $1\le\alpha\le p(x)\le \beta$ into account, we see that
\begin{align*}
\frac{1}{\lambda^\beta_\ast}\int_\Omega\left|f(x)\right|^{p(x)}\,dx&\le \int_\Omega\left|\frac{f(x)}{\lambda_\ast}\right|^{p(x)}\,dx\le\frac{1}{\lambda^\alpha_\ast}\int_\Omega\left|f(x)\right|^{p(x)}\,dx,\\
\frac{1}{\lambda^\beta_\ast}\int_\Omega\left|f(x)\right|^{p(x)}\,dx&\le 1 \le\frac{1}{\lambda^\alpha_\ast}\int_\Omega\left|f(x)\right|^{p(x)}\,dx.
\end{align*}
Hence, (see \cite{Cruz,DHHR,Zh2008a} for the details)
\begin{eqnarray}
\label{A1.1b}
\|f\|^\alpha_{L^{p(\cdot)}(\Omega)}&\le \int_\Omega |f(x)|^{p(x)}\,dx\le \|f\|^\beta_{L^{p(\cdot)}(\Omega)},\ \ if \ \|f\|_{L^{p(\cdot)}(\Omega)}>1,\nonumber\\
\|f\|^\beta_{L^{p(\cdot)}(\Omega)}&\le \int_\Omega |f(x)|^{p(x)}\,dx\le \|f\|^\alpha_{L^{p(\cdot)}(\Omega)},\ \ if\ \|f\|_{L^{p(\cdot)}(\Omega)}<1,
\end{eqnarray}
and, therefore,
\begin{align}
\label{A1.2}
\|f\|^\alpha_{L^{p(\cdot)}(\Omega)}-1&\le \int_\Omega |f(x)|^{p(x)}\,dx\le \|f\|^\beta_{L^{p(\cdot)}(\Omega)}+1,\quad\forall\, f\in L^{p(\cdot)}(\Omega),\\
\|f\|_{L^{p(\cdot)}(\Omega)}&=\int_\Omega |f(x)|^{p(x)}\,dx,\ \ if\ \ \|f\|_{L^{p(\cdot)}(\Omega)}=1.
\end{align}

The following estimates are well-known (see, for instance, \cite{Cruz,DHHR,Zh2008a}): if $f\in L^{p(\cdot)}(\Omega)$ then
\begin{align}
\label{A1.3}
\|f\|_{L^\alpha(\Omega)}&\le \left(1+|\Omega|\right)^{1/\alpha} \|f\|_{L^{p(\cdot)}(\Omega)},\\
\label{A1.4}
\|f\|_{L^{p(\cdot)}(\Omega)}&\le \left(1+|\Omega|\right)^{1/\beta^\prime}\|f\|_{L^\beta(\Omega)},\quad
\beta^\prime=\frac{\beta}{\beta-1},\quad\forall\,f\in L^\beta(\Omega).
\end{align}

Let $\left\{p_k\right\}_{k\in \mathbb{N}}\subset C^{0,\delta}(\overline{\Omega})$, with some $\delta\in(0,1]$, be a given sequence of exponents. Hereinafter in this subsection we assume that
\begin{eqnarray}
\label{A1.5.0}
p, p_k\in C^{0,\delta}(\overline{\Omega})\ \  for\ k=1,2,\dots,\  \mbox{and} \atop  p_k(\cdot)\rightarrow p(\cdot)\ \mbox{uniformly in}\ \overline{\Omega}\   \mbox{as}\  k\to\infty.
\end{eqnarray}
We associate with this sequence the following collection $\left\{f_k\in L^{p_k(\cdot)}(\Omega)\right\}_{k\in \mathbb{N}}$. The characteristic feature of this set of functions is that each element $f_k$ lives in the corresponding Orlicz space $L^{p_k(\cdot)}(\Omega)$. We say that the sequence $\left\{f_k\in L^{p_k(\cdot)}(\Omega)\right\}_{k\in \mathbb{N}}$ is bounded if
\begin{equation}
\label{A1.3.a}
\limsup_{k\to\infty}\int_{\Omega} |f_k(x)|^{p_k(x)}\,dx<+\infty.
\end{equation}

\begin{definition}
	\label{ADef 1.3}
	A bounded sequence $\left\{f_k\in L^{p_k(\cdot)}(\Omega)\right\}_{k\in \mathbb{N}}$ is weakly convergent in the variable Orlicz space $L^{p_k(\cdot)}(\Omega)$ to a function $f\in L^{p(\cdot)}(\Omega)$, where
	$p\in C^{0,\delta}(\overline{\Omega})$ is the limit of $\left\{p_k\right\}_{k\in \mathbb{N}}\subset C^{0,\delta}(\overline{\Omega})$ in the uniform topology of $C(\overline{\Omega})$, if
	\begin{equation}
	\label{A1.3.b}
	\lim_{k\to\infty} \int_{\Omega} f_k \varphi\,dx=\int_{\Omega} f \varphi\,dx,\quad \forall\, \varphi\in C^\infty_0(\mathbb{R}^N).
	\end{equation}	
\end{definition}

We make use of the following result (we refer to \cite[Lemma~13.3]{Zh2011} for comparison) concerning the lower semicontinuity property of the variable $L^{p_k(\cdot)}$-norm  with respect to the weak convergence in $L^{p_k(\cdot)}(\Omega)$.

\begin{proposition}
	\label{AProp 1.5}
	If a bounded sequence $\left\{f_k\in L^{p_k(\cdot)}(\Omega)\right\}_{k\in \mathbb{N}}$ converges weakly in $L^{\alpha}(\Omega)$  to $f$ for some $\alpha>1$, then $f\in L^{p(\cdot)}(\Omega)$, $f_k\rightharpoonup f$ in variable $L^{p_k(\cdot)}(\Omega)$, and
	\begin{equation}
	\label{A1.5}
	\liminf_{k\to\infty} \int_{\Omega} |f_k(x)|^{p_k(x)}\,dx\ge \int_{\Omega} |f(x)|^{p(x)}\,dx.
	\end{equation}
\end{proposition}

\begin{remark}
	\label{ARem 1.5}
	Arguing in a similar manner and using, instead of (\ref{A1.5.2}), the estimate
	\begin{equation*}
	\liminf_{k\to\infty} \int_{\Omega} \frac{1}{p_k(x)}|f_k(x)|^{p_k(x)}\,dx\\
	\ge
	\int_{\Omega}  f(x)\varphi(x)\,dx -
	\int_{\Omega}\frac{1}{p_k^{\prime}(x)} |\varphi(x)|^{p^\prime(x)}\,dx,
	\end{equation*}
	it can be shown that the lower semicontinuity property (\ref{A1.5}) can be generalized as follows
	\begin{equation}
	\label{A1.5.new}
	\liminf_{k\to\infty} \int_{\Omega} \frac{1}{p_k(x)}|f_k(x)|^{p_k(x)}\,dx\ge \int_{\Omega} \frac{1}{p(x)}|f(x)|^{p(x)}\,dx.
	\end{equation}
\end{remark}

The following result can be viewed as an analogous of the H\"{o}lder inequality in Lebesgue spaces with variable exponents (for the details we refer to \cite{Cruz, DHHR}).
\begin{proposition}
	\label{AProp 1.2}
	If $f\in L^{p(\cdot)}(\Omega)^N$ and $g\in L^{p^\prime(\cdot)}(\Omega)^N$, then $\left(f,g\right)\in L^1(\Omega)$ and
	\begin{equation}
	\label{AKog_1.2.1}
	\int_\Omega \left(f,g\right)\,dx\le 2\|f\|_{L^{p(\cdot)}(\Omega)^N} \|g\|_{L^{p^\prime(\cdot)}(\Omega)^N}.
	\end{equation}
\end{proposition}

%\subsection{Appendix~C.~Sobolev Spaces with Variable Exponent}

\section{Sobolev Spaces with Variable Exponent}
We recall here well-known facts concerning the Sobolev spaces with variable exponent. Let $p(\cdot)$ be a measurable exponent function on $\Omega$ such that
$1<\alpha\le p(x)\le\beta<\infty$ a.e. in $\Omega$, where $\alpha$ and $\beta$ are given constants. We associate with it the so-called Sobolev-Orlicz space
\begin{equation}
\label{B0.12}
W^{1,p(\cdot)}(\Omega):=\left\{u\in W^{1,1}(\Omega):\int_\Omega \left[|u(x)|^{p(x)}+ |\nabla u(x)|^{p(x)}\right]\, dx<\infty\right\}
\end{equation}
and equip it with the norm $\|u\|_{W^{1,p(\cdot)}(\Omega)}=\|u\|_{L^{p(\cdot)}(\Omega)}+\|\nabla u\|_{L^{p(\cdot)}(\Omega;\mathbb{R}^N)}$.

It is well-known that, in general, unlike classical Sobolev spaces, smooth functions are not necessarily dense in $W=W^{1,p(\cdot)}_0(\Omega)$. Hence, with variable exponent $p=p(x)$ ($1<\alpha\le p\le\beta$) we can associate another Sobolev space,
\[
H=H^{1,p(\cdot)}(\Omega)\ \mbox{as  the  closure  of  the  set} \ C^\infty(\overline{\Omega})\ \mbox{in} \ W^{1,p(\cdot)}(\Omega)\mbox{-norm}.
\]
Since the identity $W=H$ is not always valid, it makes sense to say that an exponent $p(x)$ is regular if $C^\infty(\overline{\Omega})$ is dense in $W^{1,p(\cdot)}(\Omega)$.

The following result reveals the important property that guarantees the regularity of exponent $p(x)$.
\begin{proposition}
	\label{BProp 1.1}
	Assume that there exists $\delta\in (0,1]$ such that $p\in C^{0,\delta}(\overline{\Omega})$. Then the set $C^\infty(\overline{\Omega})$ is dense in $W^{1,p(\cdot)}(\Omega)$, and, therefore, $W=H$.
\end{proposition}
\begin{proof}
	Let $p\in C^{0,\delta}(\overline{\Omega})$ be a given exponent. Since
	\begin{equation}
	\label{B1.5a}
	\lim\limits_{t\to 0} |t|^\delta \log(|t|)=0\quad\ with \ \delta\in(0,1],
	\end{equation}
	it follows from the H\"{o}lder continuity of $p(\cdot)$ that
	\begin{equation}
	\nonumber
	|p(x)-p(y)|\le C|x-y|^\delta
	\le \left[\sup_{x,y\in\Omega} |x-y|^\delta\log(|x-y|^{-1})\right] \omega(|x-y|),\quad\forall\,x,y\in\Omega,
	\label{B1.6}
	\end{equation}
	where $\omega(t)=C/\log(|t|^{-1})$, and $C>0$ is some positive constant.
	
	Then property (\ref{B1.5a}) implies that $p(\cdot)$ is a log-H\"{o}lder continuous function. So, to deduce the density of $C^\infty(\overline{\Omega})$ in $W^{1,p(\cdot)}(\Omega)$ it is enough to refer to Theorem 13.10 in \cite{Zh2011}.
\end{proof}


\begin{thebibliography}{99}
	
	\bibitem{Ambrosio}	
	\textsc{L.~Ambrosio, V.~Caselles, S.~Masnou and J.~M.~Morel}, \emph{The connected components of sets of finite perimeter},  \textsl{European Journal of Math.}, \textbf{3}  (2001), 39--92.
	
	\bibitem{Ballester}
	\textsc{C.~Ballester, V.~Caselles, L. Igual, J. Verdera, B. Roug\'{e}}, \textit{A Variational Model for P+XS Image Fusion}, \textsl{International Journal of Computer Vision}, \textbf{69} (2006), 43--58.	
	
	\bibitem{Blom1}
	\textsc{P.~Blomgren} \textit{Total variation methods for restoration of vector valued images}, Ph.D. Thesis, (1998), 384--387.
		
	\bibitem{Blom2}
	\textsc{P.~Blomgren, T.F.~Chan,  P.~Mulet,  C.~Wong},  \textit{Total variation image restoration: Numerical methods and extensions}, \textsl{Proceedings of the 1997 IEEE International Conference on Image Processing}, (1997), III:384--387.
	
	\bibitem{Bung1}
	\textsc{L.~Bungert, D.A.~Coomes, M.J.~Ehrhardt, J.~Rasch, R.~Reisenhofer, R. \& C.-B.~Sch\"{o}nlieb}, \textit{Blind image fusion for
		hyperspectral imaging with the directional total variation}, \textsl{Inverse
		Problems}, \textbf{34} (4) (2018), Article 044003.
	
	\bibitem{Bung2}
	\textsc{L.~Bungert, M.J.~Ehrhardt} \textit{Robust Image Reconstruction
		with Misaligned Structural Information}, \textsl{IEEE Access}, \textbf{8}  (2020), 222944--222955.
	
	
	\bibitem{CCM_99}
	\textsc{V.~Caselles, B.~Coll, J.M.~Morel},
	\textit{Topographic maps and local
		contrast changes in natural images},
	\textsl{IEEE Transactions on Image Processing}, \textbf{10}(8) (1999), 5--27.
	
	\bibitem{CCM_02}
	\textsc{V.~Caselles, B.~Coll, J.M.~Morel},
	\textit{Geometry and color in natural images},
	\textsl{J. Math. Imaging and Vision}, \textbf{16} (2002), 89--107.
	
	\bibitem{Chen}
	\textsc{Y.~Chen, S.~Levine, M.~Rao},  \textit{Variable exponent, linear growth functionals in image restoration}, \textit{SIAM Journal Appl. Math.} \textbf{66}(4), (2006), 1383--1406.
	
	\bibitem{Cruz}
	\textsc{D.V.~Cruz-Uribe, A.~Fiorenza}, \textit{Variable Lebesgue Spaces: Foundations and Harmonic Analysis}, \textsl{Birkh\"{a}user}, New York, 2013.
	
	\bibitem{CK2}
	\textsc{C.~D'Apice, U.~De~Maio, P.I.~Kogut},  \textit{An indirect approach to the existence of quasi-optimal controls in coefficients for multi-dimensional thermistor problem}, \textsl{in \textquotedblleft Contemporary Approaches and Methods in Fundamental Mathematics and Mechanics\textquotedblright}, Editors: Sadovnichiy, Victor A., Zgurovsky, Michael (Eds.). Springer. Chapter 24, (2020), 489--522.
	
	\bibitem{DAKOMA2022}
	\textsc{C.~D'Apice, P.I.~Kogut, R.~Manzo},
	\textit{On Coupled Two-Level Variational Problem in Sobolev-Orlicz Space},
	\textsl{Differential and Integral Equations}, (2022), (in press).
	
	\bibitem{DAKOKUMA2022}
	\textsc{C.~D'Apice, P.I.~Kogut, O.~Kupenko, R.~Manzo},
	\textit{On Variational Problem with Nonstandard Growth Functional and Its Applications to Image Processing},  Journal of Mathematical Imaging and
	Vision, accepted 1 Nov. 2022; https://doi.org/10.1007/s10851-022-01131-w. Publishes online.
	
	\bibitem{DAKOMAUV2022}
	\textsc{C.~D'Apice, P.I.~Kogut, R.~Manzo, M.V.~Uvarov},
	\textit{Variational Model with Nonstandard Growth Conditions for  Restoration of Satellite Optical Images Using Synthetic Aperture Radar}, \textsl{Europian Journal of Applies Math.}, Published online by Cambridge University Press: 11 March 2022, https://doi.org/10.1017/S0956792522000031.
	
	\bibitem{DAKOMAUV2021}
	\textsc{C.~D'Apice, P.I.~Kogut, R.~Manzo, M.V.~Uvarov},
	\textit{On Variational Problem with Nonstandard Growth Conditions and Its Applications to Image Processing}, \textsl{Proceeding of the 19th International Conference of Numerical Analysis and Applied Mathematics}, ICNAAM 2021, 20?26 September 2021, Location: Rhodes, Greece.
	
	\bibitem{Lions_D}
	\textsc{R.~Dautray, J.L.~Lions}, \textit{Mathematical Analysis and Numerical Methods for
		Science and Technology}, Vol.5, Springer-Verlag, Berlin Heidelberg, 1985.
	
	\bibitem{DHHR}
	\textsc{L.~Diening, P.~Harjulehto, P.~H\"{a}st\"{o}, M.~R\.{u}\^{z}i\^{c}ka}, \textit{Lebesgue and Sobolev Spaces with Variable Exponents}, \textsl{Springer}, New York, 2011.
	
	\bibitem{Frantz}	
	\textsc{D.~Frantz}, \emph{Landsat+ Sentinel-2 analysis ready data and beyond}, \textsl{Remote Sens.}, \textbf{11}, 1124, 2019.
	
	\bibitem{Gao}
	\textsc{F.~Gao, J.~Masek, M.~Schwaller, F.~Hall}, 
	\emph{On the blending of the
		Landsat and MODIS surface reflectance: Predicting daily Landsat surface
		reflectance}, \textsl{IEEE Tran. Geosci. Remote Sens.}, \textbf{44}(8) (2006), 2207--2218.
	
	\bibitem{Hilker}
	\textsc{T.~Hilker, M.~A.~Wulder, N.~C.~Coops, J.~Linke, G.~McDermid,
		J.~G.~Masek, F.~Gao, J.~C.~White}, \emph{A new data fusion model for high spatial- and temporal-resolution mapping of forest disturbance based on
		Landsat and MODIS}, \textsl{Remote Sens. Environ.}, \textbf{113}(8) (2009), 1613--1627.
	
	\bibitem{Horn}
	\textsc{B.K.~Horn, B.G.~Schunck}, 
	\textit{Determining optical flow}, \textsl{Artificial
		Intelligence}, \textbf{17}  (1981),  185--203.
	
	\bibitem{HK}
	\textsc{T.~Horsin, P.~Kogut}, 
	\textit{Optimal $L^2$-control problem in coefficients for a linear elliptic equation. I. Existence result}, 
	\textsl{Mathematical Control and Related Fields}, 
	\textbf{5} (1) (2015), 73--96.
	
	\bibitem{Jo}
	\textsc{M.V.~Joshi, K.P.~Upla}, \textit{Multi-resolution Image Fusion in Remote Sensing}, \textsl{Cambridge University Press}, Cambridge, 2019.
	
	\bibitem{Ju}
	\textsc{J.~Ju, D.~P.~Roy}, \emph{The availability of cloud-free Landsat ETM+ date over the conterminous united states and globally},  \textsl{Remote Sens. Environ.},
	\textbf{112}(3) (2008), 1196--1211.
	
	\bibitem{KhKU21}
	\textsc{P.~Khanenko, P.~Kogut, M.~Uvarov}, \textit{On Variational Problem with Nonstandard Growth Conditions for the Restoration of Clouds Corrupted Satellite Images}, \textsl{CEUR Workshop Proceedings, the 2nd International Workshop on Computational and Information Technologies for Risk-Informed Systems}, CITRisk-2021, September 16-17, 2021, Kherson, Ukraine, Volume 3101, 6--25, 2021.
	
	\bibitem{K3}
	\textsc{P.I.~Kogut},  \emph{On optimal and quasi-optimal controls in coefficients for multi-dimensional thermistor problem with mixed Dirichlet-Neumann boundary conditions}, \textsl{Control and Cybernetics}, \textbf{48}(1) (2019), 31--68.
		
	\bibitem{Kogut1}
	\textsc{V.V.~Hnatushenko, P.I.~Kogut, M.V.~Uvarov}, \emph{On flexible co-registration of optical and SAR satellite images}, in \textsl{"Lecture Notes in "Computational Intelligence and Decision Making"} (series 'Advances in Intelligent Systems and Computing'), Springer, 2021, 515--534.	
	
	\bibitem{Kogut2}
	\textsc{V.V.~Hnatushenko, P.I.~Kogut, M.V.~Uvarov}, \emph{Variational approach for rigid co-registration of optical/SAR satellite images in agricultural areas}, \textsl{Journal of Computational and Applied Mathematics}, \textbf{400}  (2022), Id 113742.
	
	\bibitem{KKM22_1}
	\textsc{P.~Kogut, Ya.~Kohut, R.~Manzo}, \emph{Fictitious Controls and Approximation of an Optimal Control Problem for Perona-Malik Equation}, \textsl{Journal of Optimization, Differential Equations and Their Applications (JODEA)}, \textbf{30} (1) (2022), 42--70.
	
	\bibitem{Kogut2022}
	\textsc{P.~Kogut, Ya.~Kohut, N.~Parfinovych},
	\textit{Solvability Issues for Some Noncoercive and Nonmonotone Parabolic Equations Arising in the Image Denoising Problems},
	\textsl{Journal of Optimization, Differential Equations and Their Applications (JODEA)}, \textbf{30} (2) (2022), 42--70.
	
	\bibitem{KKU}
	\textsc{P.I.~Kogut, O.P.~Kupenko, N.V.~Uvarov},  \emph{On increasing of resolution of satellite images via their fusion with imagery at higher resolution}, \textsl{J. of Optimization, Differential Equations and Their Applications (JODEA)}, \textbf{29}(1) (2021), 54--78.
	
	\bibitem{Lions71}
	\textsc{J.-L.~Lions},  \textit{Optimal Control of Systems Governed by Partial Differential Equations}. Springer, Berlin, 1971.
	
	\bibitem{Loncan}
	\textsc{L.~Loncan, L.B.~De Almeida, J.V.~Bioucas-Dias, X.~Briottet, J.~Chanussot, N.~Dobigeon, S.~Fabre, W.~Liao, G.A.~Licciardi, M.~Simoes, J.Y.~Tourneret, M.A.~Veganzones, G.~Vivone, Q.~Wei, N.~Yokoya}, \textit{Hyperspectral pansharpening: a review}, \textsl{IEEE Geoscience and Remote Sensing Magazine}, \textbf{3} (3) (2015), 27--46.
	
	\bibitem{Masek}
	\textsc{J.~G.~Masek, E.~F.~Vermote, N.~E.~Saleous, R.~Wolfe, F.~G.~Hall, F.~Huemmrich, F.~Gao, J.~Kutler, T.~K.~Lim}, \emph{A Landsat surface reflectance data set for North American, 1990--2000}, \textsl{IEEE Geosci. Remote 	Sens. Lett.}, \textbf{3}(1), (2006), 69--72.
	
	\bibitem{NIRENBERG}
	\textsc{L.~Nirenberg},  \emph{Topics in Nonlinear Analysis}, \textsl{Lecture Notes, New York University, New York}, 1974.
	
	\bibitem{Rani}	
	\textsc{K.~Rani, R.~Sharma}, \emph{Study of Different Image fusion Algorithm}, \textsl{International Journal of Emerging Technology and Advanced Engineering}, \textbf{3} (5) 2013, 288--290.
	
	\bibitem{Roy}
	\textsc{D.~P.~Roy, J.~Li, H.~K.~Zhang,  L.~Yan}, \emph{Best practices for the reprojection and resampling of Sentinel-2 Multi Spectral Instrument Level 1C data}, \textsl{Remote Sens. Lett.}, \textbf{7} (2016), 1023--1032.
	
	\bibitem{Roy1}
	\textsc{D.~P.~Roy, H.~Huang, L.~Boschetti, L.~Giglio, H.~K.~Zhang, J.~Li}, \emph{Landsat-8 and Sentinel-2 burned area mapping --- a combined sensor multi-temporal change detection approach}, \textsl{Remote Sens. Environ.}, \textbf{231}, 111254, 2019.
	
	\bibitem{Wang}
	\textsc{P.~Wang, F.~Gao, J.~G.~Masek}, \emph{Operational data fusion framework for building frequent Landsat-like imagery}, \textsl{IEEE Transactions on Geoscience and Remote Sensing}, \textbf{52}(11), (2014), 7353--7365.
	
	\bibitem{Yan}	
	\textsc{L.~Yan, D.~P.~Roy, H.~Zhang, J.~Li, H.~Huang}, \emph{An automated approach for sub-pixel registration of Landsat-8 Operational Land Imager (OLI) and Sentinel-2 Multi Spectral Instrument (MSI) imagery}, \textsl{Remote Sens.}, \textbf{8}, 520, 2016.
	
	\bibitem{Zh2008a}
	\textsc{V.V.~Zhikov}, \textit{Solvability of the three-dimensional thermistor problem}, \textsl{Proceedings of the Steklov Institute of Mathematics},  \textbf{281} (2008), 98--111.
	
	\bibitem{Zh2011}
	\textsc{V.V.~Zhikov}, \textit{On variational problems and nonlinear elliptic equations with nonstandard growth conditions}, \textsl{Journal of Mathematical Sciences}, \textbf{173}(5) (2011), 463--570. 
\end{thebibliography}
\end{document}